\documentclass[12pt]{article}
\usepackage[usenames,dvipsnames]{color}
\usepackage{amsmath,amsthm,verbatim,amssymb,amsfonts,amscd,diagrams, graphics, mathrsfs,hyperref}
\usepackage[matrix,arrow,curve,frame]{xy}
\theoremstyle{plain}
\newtheorem*{intro-proposition}{Proposition}
\newtheorem{theorem}{Theorem}[section]
\newtheorem{corollary}[theorem]{Corollary}
\newtheorem{lemma}[theorem]{Lemma}
\newtheorem{proposition}[theorem]{Proposition}

\theoremstyle{definition}
\newtheorem{definition}[theorem]{Definition}

\newtheorem{notation}[theorem]{Notation}

\theoremstyle{remark}
\newtheorem{remark}[theorem]{Remark}

\newarrow{ul}---->
\newarrow{Backwards}<----

\newcommand{\td}[1]{\tilde{#1}}

\newcommand{\Z}{\mathbb{Z}}

\newcommand{\R}{\mathbb{R}}

\newcommand{\bd}{\partial}

\newcommand{\ms}[1]{\mathscr{#1}}

\newcommand{\dlim}{\varinjlim}

\newcommand{\Hom}{\text{Hom}}

\newcommand{\codim}{\text{codim}}

\renewcommand{\i}{\mathfrak i}

\newarrow{Onto}----{>>}
\newarrow{Equals}=====
\newarrow{Into}C--->

\newcommand{\sW}{\sigma^*_\mathrm{Witt}}
\newcommand{\U}{{\cal U}}
\newcommand{\V}{{\cal V}}

\newcommand{\cC}{{\cal C}}
\newcommand{\cD}{{\cal D}}
\newcommand{\St}{{\mathrm{St}}}
\newcommand{\bSt}{\overline{\mathrm{St}}}
\newcommand{\Lk}{{\mathrm{Lk}}}

\newcommand\red[1]{\textcolor{black}{#1}}                         %

\begin{document}

\title{The symmetric signature of a Witt space}
\author{Greg Friedman\thanks{The first author would like to thank the Midwest
Topology Network for travel support. This work was partially supported by a grant 
from the Simons Foundation (\#209127 to Greg Friedman)} \
and James McClure\thanks{The second author was partially supported by NSF. He
thanks the Lord for making his work possible.}
}
\date{March 14, 2013}

\maketitle

\begin{abstract}
Witt spaces are pseudomanifolds for which the middle-perversity intersection
homology with rational coefficients is self-dual.  We give a new
construction of the symmetric signature for Witt spaces which is similar in
spirit to the construction given by Mi\v{s}\v{c}enko for manifolds.  Our
construction
has all of the expected properties, including invariance under stratified
homotopy equivalence.
\end{abstract}

\medskip



\tableofcontents

\section{Introduction}

For a compact oriented $m$-manifold $M$ (and more generally for a Poincar\'e
duality space) the symmetric signature $\sigma^*(M)$ is an element of the
symmetric $L$-group $L^m(\pi_1(M))$.
The symmetric signature was introduced by Mi\v{s}\v{c}enko in \cite{Miscenko} 
as a tool for studying the Novikov conjecture, and since then it has become an
important part of surgery theory (see \cite{Ra92}, for example).

The basic ingredient in the construction of $\sigma^*(M)$ is Poincar\'e
duality on the universal cover.  Another situation where Poincar\'e duality 
occurs is the middle perversity intersection homology of a certain class of 
pseudomanifolds, the {\it Witt spaces} (\cite{Si83}), so it is natural to ask 
whether there is a symmetric signature for Witt spaces.  The purpose of this 
paper is to give a positive answer to this question.  The main technical issue
is the fact that the Alexander-Whitney map, which is used in the 
construction of the symmetric signature for manifolds, does not exist for
intersection chains (because the front face and back face of an allowable
simplex need not be allowable).  We deal with this issue by showing that 
a substitute for the Alexander-Whitney map can be built from the geometric
diagonal and the cross product.

There are several other treatments of the symmetric signature for Witt spaces 
in the literature. Cappell, Shaneson, and Weinberger \cite{CSW} give a brief
description of a construction which uses the work of Quinn and Yamasaki
\cite{Q3,Yam}.  Further information is given in \cite[pages 209--210]{Wein}, 
but the complete account has not been published.  Banagl \cite[Section 4]{Ba11} 
uses the Ph.\ D.\ thesis of Thorsten Eppelmann \cite{Epp} to construct an 
$L$-homology fundamental class for a Witt space and then defines the symmetric 
signature to be the image of this class under the assembly map.  However, 
there are gaps in Eppelmann's work (Banagl, Laures and the second author are
currently working on  an improved version of \cite{Epp}, using our work as 
one ingredient).  Finally, an analytic construction of the symmetric 
signature (for smoothly stratified Witt spaces) has been given by Albin, 
Leichtnam, Mazzeo, and Piazza \cite{ALMP-combo}.

Our approach has several advantages.  It is similar in spirit to that of
Mi\v{s}\v{c}enko (and thus answers a question in \cite{ALMP-combo}).  The actual
construction uses only the diagonal map of the pseudomanifold and the cross
product on intersection chains.  The proof that this construction has the
expected properties uses the 
K\"unneth theorem of \cite{GBF20} and some basic (but nontrivial) facts about
intersection chains, which may be of interest in their own right. 
We give a simple proof of stratified homotopy
invariance; this is proved by a rather intricate analytic argument in
\cite{ALMP-combo} and it is not known how to prove it using the approach of
\cite{Ba11}.  We also give a simple proof of the product
formula; to prove this using the approach of \cite{Ba11} one would
need to show that Eppelmann's map $MIP\to L^\bullet$ is a map of ring spectra
up to homotopy.  

Another advantage of our approach is that it doesn't use the local information 
built into the definition of Witt space, so we are able to define the symmetric
signature for more general objects, which we call \emph{global Witt maps} (see Section 
\ref{m1}; an example of a global Witt map is a map from a Witt space to $B\pi$ with
$\pi$ discrete).

Applications of the symmetric signature for Witt spaces have been
given in \cite{We88,We99,Chang04}.  Also, Shmuel Weinberger has pointed out to
us that one can use the symmetric signature for Witt spaces to extend
\cite[Theorem 1.3.2]{Fowler} to Witt spaces.

An argument due to Weinberger (see \cite[Proof of Proposition 11.1]{ALMP-combo})
shows that any two definitions of the symmetric signature for Witt spaces 
must agree rationally if (1) they are bordism invariant and (2) they agree with 
Mi\v{s}\v{c}enko's definition for smooth manifolds.  Thus all of the known 
constructions of the symmetric signature agree rationally; it would be 
interesting to know whether they agree over the integers.

\medskip

Here is an outline of the paper.  In Section \ref{S: background} we review some
background from \cite{GBF25}.   For our construction of the symmetric signature
we need to know that the intersection homology version of Poincar\'e duality
for the universal cover (which is analogous to what Ranicki \cite{Ra02} calls
``universal Poincar\'e duality'') is given by a cap product; in Section \ref{S:
universal cap} we construct the cap product and in Section \ref{S: 
UPD} we give the proof of universal Poincar\'e duality.  In Section
\ref{S: SS} we give the construction of the symmetric signature for $F$-Witt
maps and we prove that it has the expected properties.  Section 6 gives some
technical facts about intersection chains which are needed for the main 
proofs. In particular we show (Proposition \ref{n3}) that an open cover of a
pseudomanifold $X$  gives, up to chain homotopy, a colimit decomposition of 
the intersection chains for any regular covering space $\tilde{X}$.  We also 
show (Proposition \ref{j7}) that for a compact PL pseudomanifold $X$ with
boundary, if
$\tilde{X}$ is a regular covering of $X$ with group $\pi$ and $F$ is a field
then the intersection chain complex of $\tilde{X}$ with coefficients in $F$
and any perversity is, up to chain homotopy, free and finitely generated over 
$F[\pi]$.

\begin{remark}
We will only consider pseudomanifolds whose top strata are oriented.  For
nonsingular manifolds one can extend the symmetric signature to the nonoriented
case by twisting with the orientation character (\cite{Miscenko,Ra92}), but the
analogous procedure does not work (at least not in any obvious way) in our
situation. The difficulty is that we need to know that cap product with the
fundamental class is an isomorphism, and for this we use
an induction over the strata, which is not applicable to 
twisted coefficients defined only on the top stratum.
\end{remark}

\begin{remark}
This paper relies quite heavily on \cite{GBF25}, to the extent that the reader is recommended to have a copy of that paper available. In particular, many of the proofs of theorems concerning universal Poincar\'e duality are quite similar in spirit, and in most of the detail, to the proofs of the corresponding ``non-universal''  theorems in \cite{GBF25}. In such cases where the additional details are transparent, rather than further clutter the current exposition, we simply refer the reader to the corresponding argument in \cite{GBF25}. 
\end{remark}

\paragraph{Acknowledgments.}  We would like to thank Paolo Piazza for
suggesting this project to us.  We would also like to thank Steve Ferry and
(especially) Shmuel Weinberger for helpful conversations and emails.

\section{Conventions and some background}\label{S: background}

We assume the reader to be conversant with intersection homology theory.  
Basic textbook introductions to intersection homology include  \cite{Bo, 
KirWoo, BaIH}, and the original papers \cite{GM1, GM2, Ki} are well worth 
reading.  We recommend \cite{GBF26} for an expository introduction to the 
version of intersection homology considered here and \cite{GBF23} for a more 
technical account. 

\paragraph{Stratified pseudomanifolds and intersection homology.}

 We note here some of our conventions, which sometimes differ from \red{those of} other 
authors. We continue the conventions of \cite{GBF25} and refer the reader 
there for more details. 

Until Section \ref{m1} 
we will work with topological stratified pseudomanifolds $X$, and thereafter
with PL stratified pseudomanifolds.  \emph{Skeleta} 
of $X$ will be denoted $X^i$. By a \emph{stratum}, we will mean a connected 
component of one of the spaces $X^i-X^{i-1}$; a stratum $Z$ is a 
\emph{singular stratum} if $\dim(Z)<\dim(X)$.  $X$ is allowed to have strata 
of codimension one unless noted otherwise.  A {\it perversity} on 
$X$ is a function from the set of strata of $X$ to $\Z$ which takes 
nonsingular strata to 0. This is a much more general definition than that in
\cite{GM1,GM2}; on the rare occasions when we want to refer to perversities as
defined in \cite{GM1,GM2} we will call them ``classical perversities.''

An {\it orientation} of a stratified pseudomanifold is a choice of orientations
for the top strata.

In the literature, there are several non-equivalent definitions of
intersection homology with general perversities.  We use the \red{singular chain version of}
\cite{GBF26, GBF23} (which is equivalent to that in \cite{Sa05}).  In
\cite{GBF26, GBF23} this version of intersection homology (with 
\red{coefficients in a field $F$}) was denoted $I^{\bar p}H_*(X;F_0)$ 
\red{and referred to as ``intersection homology with stratified coefficients''}, but (as in \cite{GBF25}) we
will denote it simply by $I^{\bar p}H_*(X;F)$ \red{and call it ``intersection homology''}.
This version of intersection homology agrees with the 
definition in \cite{GM1,GM2} when $\bar p$ is a classical perversity and $X$ 
has no strata of codimension one.

We let $D\bar p$ denote the complementary perversity 
to $\bar p$, i.e.  $D\bar p(Z)=\codim(Z)-2-\bar p(Z)$, \red{unless $Z$ is a codimension $0$ stratum, in which case $D\bar p(Z)=0$}.

We direct the reader to \cite[Section 4]{GBF25} for intersection cochains and 
for the chain-level versions of intersection (co)homology cup and cap products.

\paragraph{Signs.} 
We include a sign in the Poincar\'e duality isomorphism (see \cite[Section
4.1]{GBF18}).  Except for this we follow the
signs in \cite{Dold}, which means that we use the Koszul convention
everywhere except in the definition of the coboundary on cochains. 
Dold's convention for the differential of a cochain
(see \cite[Remark VI.10.28]{Dold}) is
\[
(\delta\alpha)(x)=-(-1)^{|\alpha|}\alpha(\bd x).
\]
This convention is necessary in order for the evaluation map to be a chain map.

\section{The cap product for covering spaces}\label{S: universal cap}

Let $p:\td X\to X$ be a regular cover with group $\pi$.  For any subset $A$ of
$X$ we write $\td A$ for $p^{-1}(A)$. \red{We assume that $\td X$ is stratified by the preimages of the strata of $X$. Note that $I^{\bar p}C_*(\td X;F)$ possesses a left $F[\pi]$-module structure induced by the geometric action of $\pi$ on $\td X$.} 

\begin{notation}
\label{m4}
\begin{enumerate}
\item
Given a perversity $\bar p$ on $X$, the perversity on $\td X$ which takes a
stratum $S$ to ${\bar p}(p(S))$ will also be denoted by $\bar p$.
\item
We will write
$I_{\bar p}\bar{C}^*(\td X;F)$ for $\Hom_{F[\pi]}( I^{\bar p}C_*(\td X;F), 
F[\pi])$ and 
$I_{\bar p}\bar{H}^*(\td X;F)$ for the cohomology groups of this complex.
\end{enumerate}
\end{notation}

\begin{remark}
\label{l3}
If the covering $p:\td X\to X$ is trivial (i.e., if it is isomorphic to the
projection $\pi\times X\to X$) then $I_{\bar p}\bar{H}^*(\td
X;F)$ is $\Hom_F(I_{\bar p}H_*(X,F),F[\pi])$.
\end{remark}

In this section we define a cap product 
$$I_{\bar q}\bar{H}^i(\td X ;F)\otimes I^{\bar r}H_j( X;F)\to 
I^{\bar p}H_{j-i}(\td X;F)$$
when $D\bar r\geq D\bar p+D\bar q$ and $F$ is a field.

The construction follows the general outline of \cite[Section 4]{GBF25}, so we
begin by constructing a suitable algebraic diagonal map.  
For a left $F[\pi]$-module $M$, let $M^t$ denote the right $F[\pi]$-module
structure on $M$ induced by the standard involution of $F[\pi]$.

Let  $$\td d:I^{\bar r}H_*(X;F)\to  H_*( I^{\bar p}C_*(\td 
X;F)^t\otimes_{F[\pi]} I^{\bar q}C_*(\td X;F))$$ 
be the composition 
\begin{align*}
I^{\bar r}H_*(X;F) &\xleftarrow{\cong} H_*(F\otimes_{F[\pi]} 
I^{\bar r}C_*(\td X;F))\\
&\xrightarrow{1\otimes d} H_*(F\otimes_{F[\pi]} I^{Q_{\bar 
p,\bar q}}C_*(\td X\times \td X;F))\\
&\xleftarrow{\cong}  H_*(F\otimes_{F[\pi]} (I^{\bar p}C_*(\td X;F)\otimes_F 
I^{\bar q}C_*(\td X;F)))\\
&\cong  H_*( I^{\bar p}C_*(\td X;F)^t\otimes_{F[\pi]} I^{\bar 
q}C_*(\td X;F)).
\end{align*}
Here $d$ is the diagonal map given by \cite[Proposition 4.2.1]{GBF25}, and \red{for perversities $\bar p$ and $\bar q$ on $\td X$, the product perversity $Q_{\bar p,\bar q}$ on $X\times X$ is defined  by
\begin{equation*}
 Q_{\bar p,\bar q}(Z\times S)=\begin{cases}
 \bar p(Z)+\bar q(S)+2, &\text{$Z,S$ both singular strata,}\\
\bar p(Z), &\text{$S$ a regular stratum and $Z$ singular,}\\
\bar q(S), &\text{$Z$ a regular stratum and $S$ singular,}\\
0,&\text{$Z,S$ both regular strata.}
 \end{cases}
 \end{equation*}}
The first  isomorphism is given by
Proposition \ref{n3}.3 below. 
The second isomorphism is given by the K\"unneth theorem
\cite[Theorem 3.1]{GBF25} and
Proposition \ref{P: tensor flat} below.  The third isomorphism is elementary.

Suppose now that $\alpha\in I_{\bar q}\bar{H}^*(\td X;F)$ and that $x\in 
I^{\bar r}H_*(X;F)$. We note that $H_*( I^{\bar p}C_*(\td X;F)^t)$ is the same
$F$-vector space as $I^{\bar p}H_*(\td X;F)$, and we
define $\alpha\smallfrown x\in I^{\bar p}H_*(\td X;F)$ by
$$\alpha \smallfrown x=(1\otimes \alpha)\td d(x).$$
Explicitly, if $\td d(x)$ is represented by a cycle $\sum_a 
y_a\otimes z_a$, then  $\alpha \smallfrown x$
is represented by $(-1)^{|\alpha||y_a|}\sum_a y_a\alpha(z_a)\in 
I^{\bar p}C_*(\td X;F)^t$.

If $\pi$ is trivial this construction reduces to  the cap product defined  in 
\cite[Section 4.3]{GBF25}.  

Similarly, when $A$ and $B$ are open subsets of $X$, we can define  the 
relative cap product 
$$I_{\bar q}\bar{H}^i(\td X,\td A;F)\otimes I^{\bar r}H_j( X, A\cup B;F)\to 
I^{\bar p}H_{j-i}(\td X, \td B;F).$$  

In the next section, we will (implicitly) use the fact that \cite[Propositions 
4.16 and 4.19]{GBF25} have analogues for the cap product discussed in this
section.  We leave it to the reader to check that the proofs in \cite{GBF25} go
through in this situation. We will also need an analogue of \cite[Proposition
4.21]{GBF25}, and for this we need to define the cohomology cross product
\[
\times:H^*(M;F)\otimes I_{\bar p}\bar{H}^*(\td X ;F)
\to
I_{\bar p}\bar{H}^*(M\times \td X;F)
\]
in the special case \red{of the product covering $M\times \td X\xrightarrow{\text{id}\times p} M\times X$, where the covering $p:\td X\to X$ is trivial and $M$ is a manifold};
we define it to be the composite
\begin{multline*}
H^*(M;F)\otimes I_{\bar p}\bar{H}^*(\td X;F)
\cong
\Hom_F(H_*(M;F),F)\otimes \Hom_F(I^{\bar p}H_*(X;F),F[\pi])
\\
\to
\Hom_F(H_*(M;F)\otimes I^{\bar p}H_*(X;F),F[\pi])
\cong
\Hom_F(I^{\bar p}H_*(M\times X;F),F[\pi])
\\
\cong
I_{\bar p}\bar{H}^*(M\times \td X;F),
\end{multline*}
using Remark \ref{l3}.

\begin{remark}
\label{l4}
This is an isomorphism when $H_*(M;F)$ is finitely generated.
\end{remark}

\section{Universal Poincar\'e duality}\label{S: UPD}

In this section, we consider ``universal'' Poincar\'e duality---the duality 
for regular coverings of stratified pseudomanifolds. For manifolds, universal 
duality plays an important role in surgery theory and in the definition of 
$L$-theory invariants, such as the symmetric signature; see \cite[Section
4.5]{Ra02} and \cite{Ra92}.

Let $F$ be a field, and 
let $X$ be an $F$-oriented $n$-dimensional stratified pseudomanifold, 
possibly noncompact. \red{Note that even if one is ultimately concerned only with compact stratified pseudomanifolds, consideration of the noncompact case is necessary for the purposes of induction within the arguments that follow.}
Let $p:\td X\to X$ be a regular cover with group $\pi$.
For each compact $K\subset X$, let $\Gamma_K$ be the fundamental 
class of $I^{\bar 0}H_n(X, X-K;F)$ (see \cite[Definition 5.9]{GBF25}) and 
let $\bar p$, $\bar q$ be complementary perversities, i.e. $\bar p(Z)+\bar q(Z)=\codim(Z)-2$ for each singular stratum $Z$.

Let 
$$\ms D:\dlim_{K} I_{\bar p}\bar{H}^i(\td X, \td X-\td K;F)  \to I^{\bar
q}H_{n-i}(\td X;F)$$
be the map obtained by passage to the direct limit from
$$\smallfrown (-1)^{in} \Gamma_K: I_{\bar p}\bar{H}^i(\td X, 
\td X-\td K;F)\to I^{\bar q}H_{n-i}(\td X;F)$$ 
(compare the discussion in \cite{GBF25} that comes before the statement of 
Theorem 6.3, and see \cite[Section 4.1]{GBF18} for the sign).

\begin{theorem}[Universal Poincar\'e duality]\label{T: universal duality}
Let $X$ be an $F$-oriented stratified pseudomanifold, possibly noncompact and 
possibly with codimension one strata, let $p:\td X\to X$ be a regular 
$\pi$-covering of $X$, and let $\bar p,\bar q$ be complementary perversities. 
Then $\ms D$ is an isomorphism.
\end{theorem}

The proof will occupy the remainder of this section.

The proof follows the same outline as the proof of \cite[Theorem 6.3]{GBF25}.
First we need the following analogue of \cite[Lemma 6.4]{GBF25}.

\begin{lemma}
\label{l5}
Let $L$ be a compact $k-1$ dimensional stratified pseudomanifold.
If the conclusion of Theorem \ref{T: universal duality} holds for $L$ with the
trivial covering map $\pi\times L\to L$ then it also holds for $cL$ with the
trivial covering map $\pi\times cL\to cL$.
\end{lemma}

The proof of this lemma is the same as that of \cite[Lemma 6.4]{GBF25}, except
that Remark \ref{l3} above should be used in place of \cite[Remark 4.9]{GBF25}.

Next we need the following analogue of \cite[Lemma 6.6]{GBF25}.

\begin{lemma}
\label{l6}
Suppose that the conclusion of Theorem \ref{T: universal duality} holds for 
the compact $F$-oriented stratified $k-1$  pseudomanifold $L$ with the 
trivial covering map $\pi\times L\to L$.  Let 
$M$ be an $F$-oriented unstratified $n-k$ manifold, and assume that
$H_*(M,M-C;F)$ is finitely generated for a cofinal collection of compact
subsets $C$.
Give $M\times cL$ the 
product stratification and the product orientation.
Then the conclusion of Theorem \ref{T: universal duality} holds for
$M\times cL$ with the trivial covering map $\pi\times M\times cL\to M\times
cL$.
\end{lemma}

The proof is the same as that of \cite[Lemma 6.6]{GBF25}, except that the
relative version of Remark
\ref{l4} above should be used in place of the relative version of \cite[Remark
4.20]{GBF25}.

The next part of the proof of \cite[Theorem 6.3]{GBF25} is a Zorn's lemma
argument using an induction over depth.  The analogous argument works in our
situation because of the following observations:

\begin{itemize}
\item
In order to construct the Mayer-Vietoris sequence for $I_{\bar p}\bar{H}^*$
it suffices to know that if $A\subset B$ are open subsets of $X$ then the 
inclusion $I^{\bar p}C_*(\td A;F)\hookrightarrow I^{\bar p}C_*(\td B;F)$ is 
split as a map of $F[\pi]$-modules.  This in turn follows from 
the proof of \cite[Proposition 2.9]{GBF10} (use the construction in that proof
with $X$ taken to be $B$ and the ordered open cover taken to be $(A,B)$).

\item
In the situation where Lemmas \ref{l5} and \ref{l6} are needed, $M\times cL$ is
contained in a distinguished neighborhood, so in particular the restriction of
the covering map $p:\td X\to X$ to a cover of  $M\times cL$ is trivial.

\item
Moreover, the $M$'s that occur in the proof are open subsets of Euclidean
space.   Such an $M$ is a PL manifold, so we can take $C$ in Lemma \ref{l6} 
to be a compact PL subspace and then $H_*(M,M-C;F)$ is finitely generated by
Poincar\'e-Lefschetz duality (\cite[Proposition VIII.7.2]{Dold}).
\end{itemize}

It remains to consider the analogues of \cite[Lemmas 6.8 and 6.9]{GBF25}.
All of the steps in the proof of \cite[Lemma 6.8]{GBF25} go through without 
change in our situation.  For the analogue of \cite[Lemma 6.9]{GBF25}, we need
to know that the map
\begin{multline*}
\lambda:
H_*(F\otimes_{F[\pi]} \dlim_{W\in \cC}\, I^{Q_{\bar p,\bar q}}C_*(\td W\times 
\td W,\td W\times (\td W-\td K\cup \td L)))
\\
\to
H_*(F\otimes_{F[\pi]}I^{Q_{\bar p,\bar q}}C_*(\td Y,\td Y-(\td X\times (\td 
K\cup \td L))))
\end{multline*}
is an isomorphism; this is immediate from Proposition \ref{n3}.2 below.


\subsection{Lefschetz duality}

For the convenience of the reader we recall
the definition of $\bd$-stratified pseudomanifolds from 
\cite[Section 7.1]{GBF25}:

\begin{definition}\label{def boundary}
An $n$-dimensional
\emph{$\bd$-stratified pseudomanifold} is a pair $(X,B)$ together with a
filtration on $X$ such that
\begin{enumerate}
\item $X-B$, with the induced filtration, is an $n$-dimensional stratified
pseudomanifold,
\item $B$, with the induced filtration, is an $n-1$ dimensional stratified
pseudomanifold,
\item\label{I: collar} $B$ has an {\it open collar neighborhood} in $X$, that is, a
neighborhood $N$ with a homeomorphism of filtered spaces $N\to
B\times [0,1)$ (where $[0,1)$ is given the trivial filtration) that takes
$B$ to $B\times \{0\}$.
\end{enumerate}
$B$ is called the 
\emph{boundary} of $X$ and denoted $\bd X$.

We will often abuse notation by referring to the ``$\bd$-stratified 
pseudomanifold $X$,'' leaving $B$ tacit.
\end{definition}

This definition includes some non-obvious subtleties, and we encourage the reader to see \cite[Section 7.1]{GBF25} for examples and further discussion.

Lefschetz duality also generalizes to the universal setting, yielding the following corollary to Theorem \ref{T: universal duality}.


\begin{theorem}[Universal Lefschetz Duality]\label{T: univ lef}
Let $X$ be an $n$-dimensional compact $\bd$-stratified
pseudomanifold  such that $X-\bd X$ is $F$-oriented. Let $p:\td X\to X$ be a
regular $\pi$-covering, and let $\bar p,\bar q$ be complementary 
perversities.  Then the cap product with $\Gamma_X$ 
gives isomorphisms 
\[
I_{\bar p}\bar{H}^i(\td X,p^{-1}(\partial X);F)\to 
I^{\bar q}H_{n-i}(\td X;F)
\]
and
\[
I_{\bar p}\bar{H}^i(\td X;F)\to 
I^{\bar q}H_{n-i}(\td X,p^{-1}(\partial X);F).
\]
\end{theorem}
\begin{proof}
 The proof follows from Theorem \ref{T: universal duality}
just as \cite[Theorem 7.10]{GBF25} follows from \cite[Theorem 6.3]{GBF25}, which itself is an adaptation of the proof in \cite{Ha} that Lefschetz duality follows from Poincar\'e duality.

The fundamental class $\Gamma_X\in I^{\bar 0}H_n(X,\bd X;F)$ is constructed in \cite[Section 7.2]{GBF25} as the image of the fundamental class $\Gamma_{X-N}\in I^{\bar 0}H_n(X-\bd X,N- \bd X;F)$, where $N$ is an open collar of $\bd X$ in $X$, under the isomorphisms $I^{\bar 0}H_n(X-\bd X,N- \bd X;F)\cong I^{\bar 0}H_n(X,N;F)\cong I^{\bar 0}H_n(X,\bd X;F)$ induced by excision and homotopy equivalence. 

Now consider the following commutative diagram
\begin{diagram}
I_{\bar p}\bar H^i(\td X-p^{-1}(\bd X),p^{-1}(N-\bd X);F)&\lTo^{\cong} &I_{\bar p}\bar H^i(\td X,p^{-1}(N);F)\\
\dTo^{(-1)^{in}\cdot\smallfrown \Gamma_{X-N}}&&\dTo^{(-1)^{in}\cdot\smallfrown \Gamma_{X-N}}\\
I^{\bar q}H_{n-i}(\td X-p^{-1}(\bd X);F)&\rTo^{\cong} &I^{\bar q}H_{n-i}(\td X;F).
\end{diagram}
On the left side, we interpret $\Gamma_{X-N}$ as an element of 
$I^{\bar 0}H_n(X-\bd X,N-\bd X;F)$, and on the right we interpret $\Gamma_{X-N}$ as an element of $I^{\bar 0}H_n(X,N;F)$.

The top isomorphism is by excision and stratified homotopy 
equivalence. The bottom isomorphism is also by stratified homotopy 
equivalence. If we take the direct limit of the diagram as $N$ shrinks to 
$\bd X$, then $\dlim I_{\bar p}\bar H^i(\td X,p^{-1}(N);F)\cong I_{\bar p}\bar H^i(\td X,p^{-1}(\bd X);F)$, while the righthand $\Gamma_{X-N}$ becomes $\Gamma_X$ (in 
fact, all maps in the directed system obtained by suitably retracting the collar are 
isomorphisms), so the righthand map becomes the map of the  first claimed isomorphism of the Theorem. On the left, the compact sets $X-N$ are cofinal among the compact sets of $X-\bd X$ so that  the left hand map becomes 
an isomorphism in the limit by Theorem \ref{T: universal duality}. It follows therefore that the right hand map also 
becomes an isomorphism in the limit, proving the theorem in this case.

For the second isomorphism, we utilize the diagram of exact sequences, which commutes up to sign,
\begin{diagram}
&\rTo &I_{\bar p}\bar{H}^i(\td X,p^{-1}(N);F)&\rTo & I_{\bar p}\bar{H}^i(\td X;F)&\rTo &I_{\bar p}\bar{H}^i(p^{-1}(N);F)&\rTo &\\
&&\dTo^{\smallfrown \Gamma_{X-N}}&&\dTo^{\smallfrown \Gamma_{X-N}}&&\dTo^{\smallfrown \bd \Gamma_{X-N}}\\
&\rTo & I^{\bar q}H_{n-i}(\td X;F)&\rTo& I^{\bar q}\bar{H}_{n-i}(\td X,p^{-1}(N);F)&\rTo &I^{\bar q}\bar{H}_{n-i-1}(p^{-1}(N);F)&\rTo &
\end{diagram}
Here $\bd \Gamma_{X-N}\in I^{\bar 0}H_{n-1}(N;F)$; it follows from the result of \cite[Section 7.2]{GBF25} that the image of $\bd \Gamma_{X-N}$ under the isomorphism $I^{\bar 0}H_{n-1}(N;F)\cong I^{\bar 0}H_{n-1}(\bd X;F)$ is the fundamental class of $\bd X$. 
Commutativity of the diagram (including the square not shown) uses the intersection homology versions of the standard properties of the cap product, such as are demonstrated in the non-universal case in \cite[Proposition 4.19]{GBF25}; the proofs in the universal case are completely equivalent. 

Now, as $N$ is a collar of $\bd X$, which we can assume to be sufficiently nice, this diagram is isomorphic via stratified homotopy equivalences to the diagram
\begin{diagram}
&\rTo &I_{\bar p}\bar{H}^i(\td X,p^{-1}(\bd X);F)&\rTo & I_{\bar p}\bar{H}^i(\td X;F)&\rTo &I_{\bar p}\bar{H}^i(p^{-1}(\bd X);F)&\rTo &\\
&&\dTo^{\smallfrown \Gamma_{X}}&&\dTo^{\smallfrown \Gamma_{X}}&&\dTo^{\smallfrown \Gamma_{\bd X}}\\
&\rTo & I^{\bar q}H_{n-i}(\td X;F)&\rTo& I^{\bar q}\bar{H}_{n-i}(\td X,p^{-1}(\bd X);F)&\rTo &I^{\bar q}\bar{H}_{n-i-1}(p^{-1}(\bd X);F)&\rTo &
\end{diagram}

Since we know that the leftmost and rightmost terms in the diagram as drawn are isomorphisms, by Theorem
\ref{T: universal duality} and by the first part of this theorem, it follows from the Five Lemma that the middle map is also an isomorphism. 
\end{proof}

\section{The symmetric signature}\label{S: SS}

In Section \ref{j19}, we review the construction of the symmetric
signature for compact oriented nonsingular manifolds.  In 
Section \ref{j20} we give an equivalent construction which
does not use the Alexander-Whitney map.  In Section \ref{m1} we define the
concept of global $F$-Witt map; an example of an global $F$-Witt map is any map from an
$F$-Witt space to $B\pi$. 
In Section \ref{j10}, we 
construct the symmetric signature for global $F$-Witt maps, and in Section 
\ref{j22} we show that it has the expected properties. 

\subsection{The symmetric signature for manifolds}

\label{j19}

Given a compact oriented manifold $M$ of dimension $m$, a discrete group 
$\pi$, and a map $f:M\to B\pi$, the symmetric signature 
$\sigma^*(f)$ is an element of the symmetric $L$-group $L^m({\mathbb 
Z}[\pi])$.  We begin by recalling the definition of this group 
from \cite[Section 1]{Ra92} (with a variation introduced in \cite{WW}), which requires some preliminary definitions.

Let $R$ be a ring with involution and let $C$ be a chain complex of left 
$R$-modules.  The involution gives a chain complex $C^t$ of right 
$R$-modules. There is a chain map, called the {\it slant product}
\[
\backslash: \Hom_R(C,R)\otimes_{\Z}(C^t\otimes_R C)\to C^t
\]
defined by 
$\alpha\otimes x\otimes y\to (-1)^{|\alpha||x|}x\alpha(y)$
(cf. 
\cite[Section VII.11]{Dold}).

\begin{definition}
\label{n16}
A chain complex $C$ over $R$ is {\it finite} if it is free and finitely 
generated over $R$ in each degree and nonzero only in finitely many degrees.
It is {\it homotopy finite} if it is chain homotopy equivalent over $R$ to a 
finite chain
complex over $R$.
\end{definition}

Let $W$ be the standard $\Z[{\mathbb Z}/2]$-free resolution of $\mathbb Z$.  Let
$\iota\in H_0( W)$ be the generator.

\begin{definition}
\label{j1}
An $n$-dimensional {\it symmetric Poincar\'e complex over $R$} is a pair 
$(C,\phi)$, where $C$ is a homotopy finite chain complex over $R$ 
and $\phi$ is a $\Z/2$-equivariant 
chain map 
\[
\phi: W\to C^t\otimes_R C
\]
which raises degrees by $n$, such that the slant product with $\phi_*(\iota)$ is
an isomorphism
\[
H^*(\Hom_R(C,R))\to H_{n-*} (C^t).
\]
\end{definition}

(Note that $H_*(C^t)=H_*(C)$) as graded abelian groups.)

\begin{definition}
\label{r1}
An $n$-dimensional {\it symmetric Poincar\'e pair over $R$} consists of a chain
map $f:C\to D$, a $\Z/2$-equivariant chain map
\[
\phi: W\to C^t\otimes_R C
\]
which raises degrees by $n-1$,
and a $\Z/2$-equivariant map of graded $R$-modules
\[
\Phi: W\to D^t\otimes_R D
\]
which raises degrees by $n$,
such that

(i) $C$ and $D$ are homotopy finite chain complexes over $R$,

(ii) $\bd\circ \Phi-(-1)^n \Phi\circ \bd$ is the composite
\[
W\xrightarrow{\phi} C^t\otimes_R C
\to
D^t\otimes_R D,
\]
and

(iii)
slant product with $\Phi$ gives an isomorphism
\[
H^*(\Hom_R(C,R))\to H_{n-*} (Cf),
\]
where $Cf$ denotes the mapping cone of $f$.
\end{definition}

\begin{definition}
\label{j+}
\begin{enumerate}
\item
Given symmetric Poincar\'e complexes $(C,\phi)$ and $(C',\phi')$ over $R$, 
the {\it direct sum} $(C,\phi)\oplus (C',\phi')$ is the symmetric Poincar\'e 
complex $(C\oplus C',\psi)$, where $\psi$ is the composite
\[
W\xrightarrow{\mathrm{diag}}
W\oplus W
\xrightarrow{\phi\oplus\phi'}
(C^t\otimes_R C)\oplus ({C'}^t\otimes_R C')
\hookrightarrow
(C\oplus C')^t\otimes_R (C\oplus C').
\]
\item
$(C,\phi)$ and $(C',\phi')$ are {\it bordant} if there is
a symmetric Poincar\'e pair $((D,\Phi),(C,\phi)\oplus (C',-\phi'))$.
\item
$L^n(R)$ is the bordism group of $n$-dimensional symmetric
Poincar\'e complexes (with addition given by direct sum).
\end{enumerate}
\end{definition}

\begin{remark}
The definition of symmetric Poincar\'e complex in \cite[Section 1]{Ra92} 
requires $C$ to be a finite chain 
complex over $R$ and not just homotopy finite.  It's easy to check (using 
the proof of \cite[Lemma 3.4]{WW}) that the $L$ groups in Definition 
\ref{j+}.3 are the same as those in \cite{Ra92}.
\end{remark}

\begin{remark}
\label{j2}
A $\Z/2$-equivariant chain map $W\to C^t\otimes_R C$ that raises degrees by 
$n$ represents an element of $H_n(\Hom_{\Z[\Z/2]}(W,C^t\otimes_R C))$.
If $(C,\phi)$ is a
symmetric Poincar\'e complex and $\psi:W\to C^t\otimes_R C$ represents the 
same homology class as
$\phi$ (i.e., if $\psi$ is $\Z/2$-equivariantly chain homotopic to $\phi$),
then $(C,\psi)$ is a 
symmetric Poincar\'e complex that is homotopy equivalent to $(C,\phi)$ 
(\cite[Definition 1.6(ii)]{Ra92}) and therefore represents the same
element of $L^n(R)$ (by 
\cite[Proposition 1.13]{Ra92}).  
\end{remark}

Now let $f:M\to B\pi$ be a map with $M$ compact oriented of dimension $n$ and
$\pi$ discrete.  
Let $\tilde{M}$ be the induced cover of $M$.
$C_*(\tilde{M})$ is homotopy finite
over $\Z[\pi]$ (for example, by \cite[Corollary 5.3]{West}).
Choose a representative $\xi\in C_n(M)$
for the fundamental class of $M$, and let $\phi_M$ be the composite
\begin{multline*}
W\cong
W\otimes \Z
\xrightarrow{1\otimes \xi}
W\otimes C_*(M)
\cong
\Z
\otimes_{\Z[\pi]}
(W \otimes C_*(\tilde{M}))
\\
\xrightarrow{1\otimes \text{EAW}}
\Z
\otimes_{\Z[\pi]}
(C_*(\tilde{M}) \otimes C_*(\tilde{M}))
\cong
(C_*(\tilde{M}))^t\otimes_{\Z[\pi]} C_*(\tilde{M}),
\end{multline*}
where EAW is the extended Alexander-Whitney 
map (which can be constructed by an acyclic models argument).
The symmetric signature $\sigma^*(f)$ is the class
in $L^n({\Z[\pi]})$ represented by the symmetric
Poincar\'e complex $(C_*(\tilde{M}),\phi_M)$.
This is independent of the choice of
$\xi$ by Remark \ref{j2}.

\begin{remark}
The usual symmetric signature of a connected compact oriented manifold $M$ is
$\sigma^*(f)$ for the map $f:M\to B\pi_1(M)$ which classifies the universal
cover of $M$.
\end{remark}

\subsection{Reformulation}

\label{j20}

In this section we give an equivalent definition of the symmetric signature
that does not use the extended Alexander-Whitney map (see Corollary 
\ref{j4}).  We use the notation of the previous section.

Our first result shows that EAW can be replaced, for our purposes, by the 
diagram
\[
W\otimes C_*(\tilde{M})
\xrightarrow{\varepsilon\otimes 1}
C_*(\tilde{M})
\xrightarrow{d}
C_*(\tilde{M}\times \tilde{M})
\xleftarrow{\times}
C_*(\tilde{M})\otimes C_*(\tilde{M}),
\]
where $\varepsilon$ is the augmentation,  $d$ is induced by the diagonal map, 
\red{and $\times$ is the singular chain cross product, sometimes called the 
shuffle product (see \cite[Exercise 12.26.2]{Dold})}.

\begin{proposition}
\label{j3}
The diagram
\[
\xymatrix{
W\otimes C_*(\tilde{M})
\ar[r]^-{\text{EAW}}
\ar[d]_{\varepsilon\otimes 1}
&
C_*(\tilde{M})\otimes C_*(\tilde{M})
\ar[d]^\times
\\
C_*(\tilde{M})
\ar[r]^-d
&
C_*(\tilde{M}\times \tilde{M})
}
\]
commutes up to $(\Z/2\times\pi)$-equivariant chain homotopy (where
$\pi$ acts diagonally on $C_*(\tilde{M})\otimes C_*(\tilde{M})$ and
$C_*(\tilde{M}\times \tilde{M})$).
\end{proposition}

We defer the proof to the end of the section.

The map
\[
C_*(\tilde{M})\otimes C_*(\tilde{M})
\xrightarrow{\times}
C_*(\tilde{M}\times \tilde{M})
\]
is a quasi-isomorphism whose domain and target are free over ${\Z[\pi]}$, hence it is a
chain homotopy equivalence over ${\Z[\pi]}$ 
(see, for example, \cite[Exercise IV.4.2]{HS}), and we obtain a quasi-isomorphism
\[
\Z\otimes_{\Z[\pi]} (C_*(\tilde{M})\otimes C_*(\tilde{M}))
\xrightarrow{1\otimes\times}
\Z\otimes_{\Z[\pi]} C_*(\tilde{M}\times \tilde{M}).
\]
This in turn induces an isomorphism
\begin{multline*}
H_*(\Hom_{\Z[\Z/2]}(W,(C_*(\tilde{M}))^t\otimes_{\Z[\pi]} C_*(\tilde{M})))
\cong
H_*(\Hom_{\Z[\Z/2]}(W,\Z\otimes_{\Z[\pi]} (C_*(\tilde{M})\otimes C_*(\tilde{M}))))
\\
\to
H_*(\Hom_{\Z[\Z/2]}(W,\Z\otimes_{\Z[\pi]} C_*(\tilde{M}\times \tilde{M}))),
\end{multline*}
which we denote by $\Upsilon$.  Let 
\[
c_f\in H_n(\Hom_{\Z[\Z/2]}(W,\Z\otimes_{\Z[\pi]} C_*(\tilde{M}\times \tilde{M})))
\]
be the class represented by the composite
\[
W
\xrightarrow{\varepsilon}
\Z
\xrightarrow{\xi}
C_*(M)
\cong
\Z\otimes_{\Z[\pi]} C_*(\tilde{M})
\xrightarrow{1\otimes d}
\Z\otimes_{\Z[\pi]} C_*(\tilde{M}\times \tilde{M}).
\]

Our next result is an easy consequence of Proposition \ref{j3}.

\begin{proposition}
$\Upsilon$ takes the homology class of $\phi_M$ to $c_f$.
\qed
\end{proposition}

Combining this with Remark \ref{j2} gives:

\begin{corollary}
\label{j4}
If $\psi:W\to (C_*\tilde{M})^t\otimes_{\Z[\pi]} C_*(\tilde{M})$
is any $\Z/2$-equivariant chain map whose homology class is $\Upsilon^{-1}(c_f)$
then $(C_*(\tilde{M}),\psi)$ is a representative for $\sigma^*(f)$.
\end{corollary}

\begin{proof}[Proof of Proposition \ref{j3}]
Let $N_*\subset C_*$ denote the subcomplex of normalized chains 
(\cite[Definition 8.3.6]{WEIB}).
Because $N_*(\tilde{M})$ and $C_*(\tilde{M})$ are free over $\Z[\pi]$, 
the quasi-isomorphism
\[
W\otimes N_*(\tilde{M})
\to
W\otimes C_*(\tilde{M})
\]
is a $\Z/2\times\pi$-equivariant chain homotopy equivalence, so it suffices to
prove commutativity of the diagram in the proposition with $C_*$ replaced by
$N_*$.  Next let $\bar{C}_*$ be the quotient of $C_*$ by the degenerate chains.
Since the map $N_*\to \bar{C}_*$ is an isomorphism (\cite[Lemma
8.3.7]{WEIB}), it suffices to prove commutativity of the diagram in the
proposition with $C_*$ replaced by $\bar{C}_*$ (this is needed for diagram
\eqref{r6} below).

We use the formula for EAW given in 
\cite[Definition 2.10(a) and Remark 2.11(a)]{McS}.
Using this, we define a \red{chain map, natural in $X$ and $Y$,
\[
\text{EEZ}: W\otimes \bar{C}_*(X\times Y)\to \bar{C}_*(X)\otimes \bar{C}_*(Y)
\] }
(here EEZ stands for ``extended Eilenberg-Zilber''),
to be the composite
\[
W\otimes \bar{C}_*(X\times Y)
\xrightarrow{EAW}
\bar{C}_*(X\times Y)\otimes \bar{C}_*(X\times Y)
\xrightarrow{(p_1)_*\otimes (p_2)_*}
\bar{C}_*(X)\otimes \bar{C}_*(Y),
\]
where $p_1$ and $p_2$ are the projections.
EAW factors as
\[
W\otimes \bar{C}_*(X)
\xrightarrow{1\otimes d}
W\otimes \bar{C}_*(X\times X)
\xrightarrow{\text{EEZ}}
\bar{C}_*(X)\otimes \bar{C}_*(X),
\]
so to prove the proposition it suffices to show that the diagram
\[
\xymatrix{
W\otimes \bar{C}_*(\tilde{M} \times \tilde{M})
\ar[r]^{\text{EEZ}}
\ar[rd]_{\varepsilon\otimes 1}
&
\bar{C}_*(\tilde{M})\otimes \bar{C}_*(\tilde{M})
\ar[d]^\times
\\
&
\bar{C}_*(\tilde{M} \times \tilde{M})
}
\]
commutes up to $(\Z/2\times \pi_1 M)$-equivariant chain homotopy.
Since the map
\[
W\otimes \bar{C}_*(\tilde{M})\otimes \bar{C}_*(\tilde{M})
\xrightarrow{1\otimes \times}
W\otimes \bar{C}_*(\tilde{M} \times \tilde{M})
\]
is a $(\Z/2\times \pi_1 M)$-equivariant chain homotopy equivalence, it suffices
to show that the composites 
\[
W\otimes \bar{C}_*(\tilde{M})\otimes \bar{C}_*(\tilde{M})
\xrightarrow{1\otimes \times}
W\otimes \bar{C}_*(\tilde{M} \times \tilde{M})
\xrightarrow{\text{EEZ}}
\bar{C}_*(\tilde{M})\otimes \bar{C}_*(\tilde{M})
\xrightarrow{\times}
\bar{C}_*(\tilde{M} \times \tilde{M})
\]
and
\[
W\otimes \bar{C}_*(\tilde{M})\otimes \bar{C}_*(\tilde{M})
\xrightarrow{1\otimes \times}
W\otimes \bar{C}_*(\tilde{M} \times \tilde{M})
\xrightarrow{\varepsilon\otimes 1}
\bar{C}_*(\tilde{M} \times \tilde{M})
\]
are equal.  
The second composite is equal to 
\[
W\otimes \bar{C}_*(\tilde{M})\otimes \bar{C}_*(\tilde{M})
\xrightarrow{\varepsilon\otimes 1}
\bar{C}_*(\tilde{M})\otimes \bar{C}_*(\tilde{M})
\xrightarrow{\times}
\bar{C}_*(\tilde{M} \times \tilde{M}),
\]
so it suffices to show that
the diagram 
\begin{equation}
\label{r6}
\xymatrix{
W\otimes \bar{C}_*(X)\otimes \bar{C}_*(Y)
\ar[r]^-{1\otimes \times}
\ar[rd]_{\varepsilon\otimes 1}
&
W\otimes \bar{C}_*(X\times Y)
\ar[d]^{\text{EEZ}}
\\
&
\bar{C}_*(X)\otimes \bar{C}_*(Y)
}
\end{equation}
commutes, and this is easily checked from the definitions of EEZ and $\times$.
\end{proof}

\subsection{Global $F$-Witt maps.}

\label{m1}

For the next three sections we use coefficients in a field $F$.

Recall that the upper middle perversity $\bar{n}$ is defined by
\[
\bar{n}(Z)=
\begin{cases}
0, & \text{if $\codim(Z)\leq 1$}, \\
\lceil \frac{\codim(Z)-2}{2}\rceil, & \text{if $\codim(Z)\geq 2$}.
\end{cases}
\]
The lower middle perversity $\bar{m}$ is defined to be $D\bar{n}$.

\begin{definition}
\label{m2}
Let $X$ be a compact oriented 
stratified PL pseudomanifold with no strata of codimension one, and let $\pi$ 
a discrete group.  A map
$f:X\to B\pi$ is a {\it global $F$-Witt map} if
the natural map
\[
I^{\bar m}H_*(\tilde{X};F) 
\to
I^{\bar n}H_*(\tilde{X};F)
\]
(where $\tilde{X}$ is the cover induced by $f$)
is an isomorphism.
\end{definition}

For example, if $X$ is a compact $F$-Witt space (see \cite[Section 5.6.1]{GM2})
then every map $X\to B\pi$ is a global
$F$-Witt map; this is because the sheaf map 
\[
{\mathbf I}^{\bar m}{\mathbf C}_*(X;F)\to
{\mathbf I}^{\bar n}{\mathbf C}_*(X;F)
\]
is a quasi-isomorphism and thus the sheaf map
\[
{\mathbf I}^{\bar m}{\mathbf C}_*(\tilde{X};F)\to
{\mathbf I}^{\bar n}{\mathbf C}_*(\tilde{X};F)
\]
is also a quasi-isomorphism, so that 
\[
I^{\bar m}H_*(\tilde{X};F) 
\to
I^{\bar n}H_*(\tilde{X};F)
\]
is an isomorphism as required.  But notice that the condition in Definition
\ref{m2} refers only to global sections and is much weaker than a
quasi-isomorphism of sheaves.

\begin{remark}
The assumption in Definition \ref{m2} that $X$ has a PL structure
will only be used in 
the proof of Proposition \ref{j7}. 
It seems likely that this 
result does not actually require a PL structure, but the proof would be harder.
\end{remark}

We also need the analogue of Definition \ref{m2} when $X$ has a boundary.

\begin{definition}
Let $X$ be a compact oriented PL pseudomanifold with boundary and $\pi$ a 
discrete group.  A map $f:X\to B\pi$ is an {\it global $F$-$\partial$-Witt map} 
if the natural map
\[
I^{\bar m}H_*(\tilde{X};F) 
\to
I^{\bar n}H_*(\tilde{X};F) 
\]
is an isomorphism.
\end{definition}

\begin{proposition}
\label{m3}
Let $X$ be a compact oriented PL pseudomanifold with boundary and let
$f:X\to B\pi$ be a global $F$-$\partial$-Witt map.  Give $\partial X$ the orientation
inherited from $X$. Then the restriction $f|_{\partial X}$ is a global $F$-Witt map.
\end{proposition}

For the proof we need a technical fact about intersection chains which will be
proved in Section \ref{j9}.

\begin{proposition}
\label{j7}
Let $X$ be a compact PL pseudomanifold with boundary.  
Let $\tilde{X}$ be a regular covering of $X$ with group $\pi$.
For any perversity 
$\bar p$, the chain complex $I^{\bar p}C_*(\tilde{X};F)$ is homotopy finite
over $F[\pi]$.
\end{proposition}

\begin{proof}[Proof of Proposition \ref{m3}]
By Proposition \ref{j7} and \cite[Exercise IV.4.2]{HS}, the map
\[
I^{\bar m} C_*(\tilde{X};F)
\to
I^{\bar n} C_*(\tilde{X};F)
\]
is a chain homotopy equivalence over $F[\pi]$.  Therefore (using Notation
\ref{m4}.2) the map
\[
I_{\bar n} \bar{H}^*(\tilde{X};F)
\to
I_{\bar m} \bar{H}^*(\tilde{X};F)
\]
is an isomorphism. 
Let $p:\tilde{X}\to X$ be the covering induced by $f$.
Then Theorem \ref{T: univ lef} implies that the map 
\[
I^{\bar m} H_*(\tilde{X},p^{-1}(\partial X);F)
\to
I^{\bar n} H_*(\tilde{X},p^{-1}(\partial X);F)
\]
is an isomorphism.
Now the five lemma shows that the map
\[
I^{\bar m} H_*(p^{-1}(\partial X);F)
\to
I^{\bar n} H_*(p^{-1}(\partial X);F)
\]
is an isomorphism as required.
\end{proof}

\subsection{Definition of the symmetric signature for global $F$-Witt maps}

\label{j10}

Let $f:X\to B\pi$ be a global $F$-Witt map, and let  $\tilde{X}$ be the induced 
cover of $X$.
In order to define the symmetric signature of $f$ we follow the pattern of 
Section \ref{j20}:

By Proposition \ref{P: tensor flat}, the map 
\[
F\otimes_{F[\pi]} (I^{\bar n}C_*(\td X;F)\otimes_F
I^{\bar n}C_*(\td X;F))
\xrightarrow{1\otimes \times}
F\otimes_{F[\pi]} I^{Q_{\bar n ,\bar n}}C_*(\td X\times \td X;F)
\]
is a  quasi-isomorphism, which is evidently $\Z/2$-equivariant \red{with 
respect to the $\Z/2$-actions that permute the factors of $I^{\bar n}C_*(\td X;F)\otimes_F
I^{\bar n}C_*(\td X;F)$ and $\td X\times \td X$}.  
Combining this with the isomorphism
\[
(I^{\bar n}C_*(\td X;F))^t\otimes_{F[\pi]} I^{\bar
n}C_*(\td
X;F)
\cong
F\otimes_{F[\pi]} (I^{\bar n}C_*(\td X;F)\otimes_F
I^{\bar n}C_*(\td X;F)),
\]
we obtain an isomorphism
\begin{multline*}
\Upsilon: 
H_*(\Hom_{\Z[\Z/2]}(W,(I^{\bar n}C_*(\td X;F))^t\otimes_{F[\pi]} I^{\bar 
n}C_*(\td
X;F)))
\\
\xrightarrow{\cong}
H_*(\Hom_{\Z[\Z/2]}(W,F\otimes_{F[\pi]} I^{Q_{\bar n ,\bar n}}C_*(\td X\times
\td X;F))).
\end{multline*}

Next we construct a class 
\[
c_f\in H_n(\Hom_{\Z[\Z/2]}(W,F\otimes_{F[\pi]} I^{Q_{\bar n,
\bar n}}C_*(\td X\times \td X;F))).
\]

\begin{definition}
\label{j17}
(i)
Let $b_X\in H_n( F\otimes_{F[\pi]} I^{\bar 0}C_*(\td X;F))$ map to the
fundamental class $\Gamma_X$ under the isomorphism 
$$H_*(F\otimes_{F[\pi]} I^{\bar 0}C_*(\td{X};F))
\to
I^{\bar 0}H_*(X;F)$$
given by Proposition \ref{n3}.3.  

(ii)
Let
$\zeta_X$ be a cycle representing $b_X$ and let 
$c_f$ 
be the class represented by the composite
\[
W\xrightarrow{\varepsilon}
\Z
\xrightarrow{\zeta_X}
F\otimes_{F[\pi]} I^{\bar 0}C_*(\td X;F)
\xrightarrow{1\otimes d}
F\otimes_{F[\pi]} I^{Q_{\bar n ,\bar n}}C_*(\td X\times
\td X;F).
\]
\end{definition}

\begin{proposition}
\label{j6}
Let 
\[
\psi: W\to (I^{\bar n}C_*(\td X;F))^t\otimes_{F[\pi]} I^{\bar
n}C_*(\td X;F)
\]
be a $\Z/2$-equivariant chain map that represents $\Upsilon^{-1}(c_f)$.  
Then 
\[
(I^{\bar n}C_*(\td X;F),\psi)
\] 
is a symmetric Poincar\'e complex.
\end{proposition}

Before proving this we give

\begin{definition}
\label{j8}
The {\it symmetric signature of a global $F$-Witt map $f:X\to B\pi$}, denoted 
$\sW(f)$, is the class in $L^n(F[\pi])$ represented  by $(I^{\bar n}C_*(\td
X;F),\psi)$, with $\psi$ as in Proposition \ref{j6}.
\end{definition}

\begin{proof}[Proof of Proposition \ref{j6}]
First we note that $I^{\bar n}C_*(\td X;F)$ is homotopy finite over $F[\pi]$ 
by Proposition \ref{j7}.  It remains to show
that the slant product
with $\psi_*(\iota)$ induces an isomorphism 
\[
H^*(\Hom_{F[\pi]}(I^{\bar n}C_*(\td X;F),F[\pi]))
\to
H_{n-*}(I^{\bar n}C_*(\td X;F)).
\]
Consider the diagram
\[
\xymatrix{
H^*(\Hom_{F[\pi]}(I^{\bar n}C_*(\td X;F),F[\pi]))
\ar[r]^-{\backslash{\td d}(\Gamma_X)}
\ar[rd]_{\backslash \psi_*(\iota)}
&
H_{n-*}(I^{\bar m}C_*(\td X;F))
\ar[d]
\\
&
H_{n-*}(I^{\bar n}C_*(\td X;F)).
}
\]
The map $\td{d}$ was defined in Section \ref{S: universal cap}. 
The vertical arrow is 
an isomorphism by the definition of global $F$-Witt map
and the horizontal arrow is an isomorphism by Theorem \ref{T: universal 
duality}, so it suffices to show that the diagram commutes.

For this it suffices to show that the lower horizontal arrow in the following
commutative diagram takes ${\td
d}(\Gamma_X)$
to $\psi_*(\iota)$.
\[
\xymatrix{
H_*(F\otimes_{F[\pi]} I^{\bar 0}C_*(\td X;F))
\ar[d]^{1\otimes d}
\ar[r]^=
&
H_*(F\otimes_{F[\pi]} I^{\bar 0}C_*(\td X;F))
\ar[d]^{1\otimes d}
\\
H_*(F\otimes_{F[\pi]} I^{Q_{\bar n ,\bar m}}C_*(\td X\times
\td X;F))
\ar[r]
&
H_*(F\otimes_{F[\pi]} I^{Q_{\bar n ,\bar n}}C_*(\td X\times
\td X;F))
\\
H_*(F\otimes_{F[\pi]} (I^{\bar n}C_*(\td X;F)\otimes_F
I^{\bar m}C_*(\td X;F)))
\ar[u]^\cong_{1\otimes\times}
\ar[d]_\cong
\ar[r]
&
H_*(F\otimes_{F[\pi]} (I^{\bar n}C_*(\td X;F)\otimes_F
I^{\bar n}C_*(\td X;F)))
\ar[u]^\cong_{1\otimes\times}
\ar[d]_\cong
\\
H_*((I^{\bar n}C_*(\td X;F))^t\otimes_{F[\pi]} I^{\bar m}C_*(\td X;F))
\ar[r]
&
H_*((I^{\bar n}C_*(\td X;F))^t\otimes_{F[\pi]} I^{\bar n}C_*(\td X;F))
}
\]
The definition of $\td{d}$ shows that ${\td
d}(\Gamma_X)$ is the image of $b_X$ (see Definition \ref{j17}(i)) under the 
left vertical composite, and the definition of $\psi$ shows that 
$\psi_*(\iota)$ is the image of $b_X$ under the right vertical composite, 
which completes the proof.
\end{proof}

\subsection{Properties of the symmetric signature for global $F$-Witt maps}
\label{j22}

We begin by showing (Proposition \ref{j12}) 
that $\sW$ is consistent with the usual symmetric signature $\sigma^*$ when $X$ is a manifold and 
(Proposition \ref{j13}) that $\sW$ is consistent
with the Witt class $w$ of $X$, as defined in \cite[Section I.4]{Si83} and 
\cite[Section 4.1]{GBF21}.  We then show that $\sW$ is additive with respect
to disjoint union (Proposition \ref{j14}) and multiplicative with respect to
Cartesian product (Theorem \ref{j15}).  
Next we show that $\sW$ is
invariant under oriented homeomorphism and oriented stratified
homotopy equivalence (Theorem \ref{j16}), and bordism of global $F$-Witt maps 
(Theorem \ref{j18}).
Finally, we note that $\sW$ agrees rationally with the 
signature index class constructed in \cite[Theorem 1.1]{ALMP-combo}; it would be
interesting to know whether they also agree integrally.

\begin{proposition}
\label{j12}
If $X$ is a compact oriented manifold and $f:X\to B\pi$ is any map then 
$\sW(f)$ is equal to the usual symmetric signature $\sigma^*(f)$. 
\end{proposition}

\begin{proof} 
This is immediate from Corollary \ref{j4}.
\end{proof}

For our next result, we recall that
there is a map $L^n(F[\pi])\to L^n(F)$ which takes the class
of 
\[
(C,W\xrightarrow{\phi} C^t\otimes_{F[\pi]} C)
\]
to the class of\footnote{Here $C/\pi$ is a convenient shorthand for $F\otimes_{F[\pi]} C$, where $F$ is given the trivial $\pi$ action.} 
\[
(C/\pi, W\xrightarrow{\phi} C^t\otimes_{F[\pi]} C
\to C/\pi \otimes_F\, C/\pi ).
\]
Moreover, if $n\equiv 0$ mod 4, or if char$(F)=2$ and $n\equiv 0$ mod 2, we can
construct a map $L^n(F)\to W(F)$ (where $W(F)$ is the Witt group) as follows:
a symmetric Poincar\'e complex $(D, \psi)$ over $F$ determines an inner product
\[
H_{n/2}(\Hom_F(D,F))\otimes_FH_{n/2}(\Hom_F(D,F))\to F
\]
which takes $\alpha\otimes \beta$ to $(\alpha\otimes \beta)(\psi_*(\iota))$,
and the proof of \cite[Proposition
VIII.9.6]{Dold} shows that the element of $W(F)$ represented by this inner
product depends only on the bordism class of $(D, \psi)$.

\begin{proposition}
\label{j13}
Let $f:X\to B\pi$ be a global $F$-Witt map.  Let $n$ be the dimension of $X$ and
suppose that
$n\equiv 0$ mod 4 or $\mathrm{char}(F)=2$ and $n\equiv 0$ mod 2.
Then the composite
\[
L^{n}(F[\pi])\to L^{n}(F)\to W(F)
\]
takes $\sW(f:X\to B\pi)$ to the Witt class $w(X)$ (that is, the class of the 
intersection form on $I^{\bar n}H_{n/2}(X;F)$).
\end{proposition}

\begin{proof}
Let $(I^{\bar n}C_*(\td X;F),\psi)$ be a representative for $\sW(f:X\to B\pi)$,
where $\psi$ satisfies the condition of Proposition \ref{j6}.
The image of $\sW(f:X\to B\pi)$ in $L^n(F)$ is represented by 
$(I^{\bar n}C_*(\td X;F)/\pi, \omega)$,
where $\omega$ is the composite
\[
W\xrightarrow{\psi} I^{\bar n}C_*(\td X;F)^t\otimes_{F[\pi]}
I^{\bar n}C_*(\td X;F)
\to
I^{\bar n}C_*(\td X;F)/\pi\otimes_F I^{\bar n}C_*(\td X;F)/\pi.
\]
Let $\omega'$ be the composite
\[
W\xrightarrow{\omega}
I^{\bar n}C_*(\td X;F)/\pi\otimes_F I^{\bar n}C_*(\td X;F)/\pi
\to
I^{\bar n}C_*(X;F)\otimes_F I^{\bar n}C_*(X;F).
\]
The map $I^{\bar n}C_*(\td X;F)/\pi\to I^{\bar n}C_*(X;F)$ is a
chain homotopy equivalence by Proposition \ref{n3}.3, and so $(I^{\bar
n}C_*(X;F),\omega')$ is bordant to $(I^{\bar n}C_*(\td X;F)/\pi, \omega)$ by 
\cite[Proposition 1.13 and Definition 1.6(ii)]{Ra92}.
It is straightforward to check that $\omega'_*(\iota)$ is the element
$\bar d (\Gamma_X)$, where $\bar d$ is the algebraic diagonal defined in
\cite[Section 4.1]{GBF25}.  Thus the image of $\sW(f:X\to B\pi)$ in $W(F)$ is 
represented by the inner product
\[
I_{\bar n}H^{n/2}(X;F)\otimes_F I_{\bar n}H^{n/2}(X;F)\to F
\]
which takes $\alpha\otimes \beta$ to $(\alpha\otimes \beta)\bar d
(\Gamma_X)=(\alpha\cup\beta)(\Gamma_X)$. By \cite[Proposition 4.19]{GBF25} 
and the main result of \cite{GBF30} we see that the Poincar\'e duality 
isomorphism $I_{\bar n}H^{n/2}(X;F)\to I^{\bar m}H_{n/2}(X;F)$ takes this 
inner product to the intersection form.
\end{proof}

\begin{proposition}
\label{j14}
If $f:X\to B\pi$ and 
$g:Y\to B\pi$ are global $F$-Witt maps and $X$ and $Y$ have the same dimension then
\[
\sW(f\textstyle{\coprod} g:X\textstyle{\coprod} Y\to B\pi)=\sW(f)+\sW(g).
\]
\end{proposition}

\begin{proof}
The proof is a straightforward diagram chase using Definitions \ref{j+} and
\ref{j8}.
\end{proof}

Next recall the multiplication map
\[
L^m(F[\pi])\otimes L^n(F[\rho])\to L^{m+n}(F[\pi\times \rho])
\]
(see \cite[Proposition 8.1]{Ra80a}).  Also recall that $B\pi \times
B\rho$ is canonically homeomorphic to $B(\pi\times \rho)$.

\begin{proposition}
\label{j15}
If $f:X\to B\pi$, $g:Y\to B\rho$ are global $F$-Witt maps then
\[
\sW(f\times g)=\sW(f)\cdot\sW(g).
\]
\end{proposition}

\begin{proof}
First we note that
$\widetilde{X\times Y}$ is canonically homeomorphic to $\tilde{X} \times 
\tilde{Y}$.

Let $(I^{\bar n}C_*(\td{X};F),\psi_X)$ and 
$(I^{\bar n}C_*(\td{Y};F),\psi_Y)$ be representatives for $\sW(f)$ and 
$\sW(g)$.  Recall the map 
\[
\Delta:W\to W\otimes W
\]
defined on page 174 of \cite{Ra80a}.
The product $\sW(f)\cdot\sW(g)$ is (by definition) represented by 
\[
(I^{\bar n}C_*(\td{X};F)
\otimes_F I^{\bar n}C_*(\td{Y};F), \omega),
\]
where $\omega$ is the composite
\begin{multline*}
W\xrightarrow{\Delta}
W\otimes W
\xrightarrow{\phi\otimes\psi}
(I^{\bar n}C_*(\td{X};F)^t\otimes_{F[\pi]} I^{\bar n}C_*(\td{X};F))
\otimes
(I^{\bar n}C_*(\td{Y};F)^t\otimes_{F[\rho]} I^{\bar n}C_*(\td{Y};F))
\\
\cong
(I^{\bar n}C_*(\td{X};F)\otimes I^{\bar n}C_*(\td{Y};F))^t
\otimes_{F[\pi\times\rho]}
(I^{\bar n}C_*(\td{X};F)\otimes I^{\bar n}C_*(\td{Y};F)).
\end{multline*}
By \cite[Proposition 1.13]{Ra92} it suffices to show that there
is a homotopy equivalence (\cite[Definition 1.6(ii)]{Ra92}) from
$(I^{\bar n}C_*(\td{X};F)
\otimes_F I^{\bar n}C_*(\td{Y};F), \omega)$ to a representative for
$\sW(f\times g)$.
By \cite[page 382]{GBF20}, the cross product induces a map\footnote{While we 
have been using cross products of the form $I^{\bar p}C_*(A;F)\otimes_F 
I^{\bar p}C_*(B;F)\to I^{Q_{\bar p,\bar q}}C_*(A\times B;F)$, 
it follows  more 
generally from the argument  in \cite[page 382]{GBF20} that the cross product 
$I^{\bar p}C_*(A;F)\otimes_F I^{\bar p}C_*(B;F)\to I^{Q}C_*(A\times B;F)$
exists
(but is not necessarily a quasi-isomorphism) whenever $Q(S\times T)\geq 
\bar p(S)+\bar q(T)$ for any strata $S\subset A$ and $T\subset B$. In 
particular, this holds true when  $Q=\bar p=\bar q=\bar n$. } 
\[
I^{\bar n}C_*(\td{X};F)
\otimes_F I^{\bar n}C_*(\td{Y};F)
\to
I^{\bar n}C_*(\td{X}\times \td{Y};F),
\]
and this is a quasi-isomorphism by \cite[Corollary 3.6]{GBF20} and its proof.
To show that this map gives the desired homotopy equivalence of symmetric 
Poincar\'e complexes it suffices to show that the composite
\begin{multline}
\label{r2}
W\xrightarrow{\omega}
(I^{\bar n}C_*(\td{X};F)
\otimes_F I^{\bar n}C_*(\td{Y};F))^t\otimes_{F[\pi\times \rho]}
(I^{\bar n}C_*(\td{X};F)
\otimes_F I^{\bar n}C_*(\td{Y};F))
\\
\xrightarrow{\times\otimes\times}
I^{\bar n}C_*(\td{X}\times \td{Y};F)^t
\otimes_{F[\pi\times \rho]}
I^{\bar n}C_*(\td{X}\times \td{Y};F)
\\
\cong
F\otimes_{F[\pi\times \rho]} (I^{\bar n}C_*(\td{X}\times \td{Y};F)
\otimes_F I^{\bar n}C_*(\td{X}\times \td{Y};F))
\xrightarrow{1\otimes \times}
F\otimes_{F[\pi\times \rho]} (I^{Q_{\bar n ,\bar n}}
C_*(\td{X}\times \td{Y}\times \td{X}\times \td{Y};F))
\end{multline}
represents the homology class $c_{f\times g}$ defined in Definition
\ref{j17}(ii).

In order to do this,  we first define a perversity $R$ on $\td{X}\times 
\td{X}\times \td{Y}\times \td{Y}$ by 
\[
R(S\times T\times U\times V)
=\begin{cases}
Q_{\bar n ,\bar n}(S\times T)
+
Q_{\bar n ,\bar n}(U\times V)
&
\text{if $S\times T$ or $U\times V$ is nonsingular},\\
Q_{\bar n ,\bar n}(S\times T)
+
Q_{\bar n ,\bar n}(U\times V)
+2
&
\text{otherwise}.
\end{cases}
\]
Then
\[
Q_{\bar n ,\bar n}(S\times U\times T\times V)\leq
R(S\times T\times U\times V) 
\]
for all quadruples $S,T,U,V$, so there is an inclusion map
\[
i: I^{Q_{\bar n ,\bar n}}
C_*(\td{X}\times \td{Y}\times \td{X}\times \td{Y};F)
\to
I^R C_*(\td{X}\times \td{X}\times \td{Y}\times \td{Y};F).
\]
Consider the diagram
\[
\xymatrix{
I^{\bar n}C_*(\td{X};F)
\otimes
I^{\bar n}C_*(\td{X};F)
\otimes
I^{\bar n}C_*(\td{Y};F)
\otimes
I^{\bar n}C_*(\td{Y};F)
\ar[r]^-\times
\ar[dd]_\times
&
I^{\bar n}C_*(\td{X}\times \td{Y};F)
\otimes
I^{\bar n}C_*(\td{X}\times \td{Y};F)
\ar[d]_\times
\\
&
I^{Q_{\bar n ,\bar n}}
C_*(\td{X}\times \td{Y}\times \td{X}\times \td{Y};F)
\ar[d]_i
\\
I^{Q_{\bar n ,\bar n}}
C_*(\td{X}\times \td{X};F)
\otimes
I^{Q_{\bar n ,\bar n}}
C_*(\td{Y}\times \td{Y};F)
\ar[r]^-\times
&
I^R C_*(\td{X}\times \td{X}\times \td{Y}\times \td{Y};F).
}
\]
All of the cross products are quasi-isomorphisms,%
\footnote{For the lower horizontal arrow, this follows from the proof of
\cite[Theorem 5.4]{GBF20}; full details of this argument will be given in 
\cite{GregBook}.}
so the map $i$ is also a quasi-isomorphism.  

It therefore suffices to show that the composite of \eqref{r2} with 
$1\otimes i$ represents $(1\otimes i)_*c_{f\times g}$.  If $\chi_f$ and 
$\chi_g$ are representatives of
$c_f$ and $c_g$ then (by the definition of $\sW(f)$ and $\sW(g)$)
the composite of \eqref{r2} with $1\otimes i$ is $\Z/2$-equivariantly chain 
homotopic to the composite
\begin{multline}
\label{r3}
W\to W\otimes W
\xrightarrow{\chi_f\otimes \chi_g}
(F\otimes_{F[\pi]} I^{Q_{\bar n ,\bar n}}C_*(\td X\times \td X;F))
\otimes
(F\otimes_{F[\pi]} I^{Q_{\bar n ,\bar n}}C_*(\td Y\times \td Y;F))
\\
\to
F\otimes_{F[\pi\times \rho]}
( I^{Q_{\bar n ,\bar n}}C_*(\td X\times \td X;F)
\otimes
I^{Q_{\bar n ,\bar n}}C_*(\td Y\times \td Y;F))
\xrightarrow{1\otimes \times}
F\otimes_{F[\pi\times \rho]}
I^RC_* (\td{X}\times \td{X}\times \td{Y}\times\td{Y};F).
\end{multline}
By Definition \ref{j17}(ii), the composite \eqref{r3} is $\Z/2$ equivariantly
chain homotopic to the composite
\begin{multline}
\label{r4}
W\to
W\otimes W
\to
F\otimes F
\xrightarrow{\zeta_f\otimes\zeta_g}
(F\otimes_{F[\pi]}
I^{\bar 0}C_*(\td{X};F))
\otimes_F
(F\otimes_{F[\rho]}
I^{\bar 0}C_*(\td{Y};F))
\\
\cong
F\otimes_{F[\pi\times \rho]}
(I^{\bar 0}C_*(\td{X};F)
\otimes_F
I^{\bar 0}C_*(\td{Y};F))
\xrightarrow{1\otimes \times}
F\otimes_{F[\pi\times \rho]}
I^{\bar 0}C_*(\td{X}\times\td{Y};F)
\\
\xrightarrow{1\otimes d}
F\otimes_{F[\pi\times \rho]}
I^{Q_{\bar n,\bar n}}C_*(\td{X}\times\td{Y}\times\td{X}\times\td{Y};F)
\xrightarrow{1\otimes i}
F\otimes_{F[\pi\times \rho]}
I^RC_*(\td{X}\times\td{X}\times\td{Y}\times\td{Y};F).
\end{multline}
By the definition of $c_{f\times g}$ it now suffices to show that a piece of 
the composite \eqref{r4}, namely
\begin{multline*}
F\otimes F
\xrightarrow{\zeta_f\otimes\zeta_g}
(F\otimes_{F[\pi]}
I^{\bar 0}C_*(\td{X};F))
\otimes_F
(F\otimes_{F[\rho]}
I^{\bar 0}C_*(\td{Y};F))
\\
\cong
F\otimes_{F[\pi\times \rho]}
(I^{\bar 0}C_*(\td{X};F)
\otimes_F
I^{\bar 0}C_*(\td{Y};F))
\xrightarrow{1\otimes \times}
F\otimes_{F[\pi\times \rho]}
I^{\bar 0}C_*(\td{X}\times\td{Y};F)
\end{multline*}
is chain homotopic to the composite
\[
F\otimes F
\to 
F
\xrightarrow{\zeta_{f\otimes g}}
F\otimes_{F[\pi\times \rho]}
I^{\bar 0}C_*(\td{X}\times\td{Y};F),
\]
and this in turn is a straightforward consequence of the fact that the
fundamental class $\Gamma_{X\times Y}\in I^{\bar 0}C_*(X\times Y;F)$ is the
cross product $\Gamma_X\times \Gamma_Y$ (which follows from \cite[Corollary
5.16]{GBF25}).
\end{proof}

Next we need to recall some results from \cite{GBF25}.  Let $X$ and $Y$
be compact 
$n$-dimensional stratified pseudomanifolds, and let $X$ be oriented (recall
that an orientation of $X$ is just an orientation, in the usual sense, of the
top stratum of $X$).  

Let $g:Y\to X$ be a 
homeomorphism of the underlying topological spaces, {\it not} necessarily 
compatible with the stratifications. 
The stratification of $X$ induces a restratification 
$Y'$ of $Y$ (with $(Y')^i=g^{-1}(X^i)$) and the orientation of $X$ induces an
orientation of $Y'$.  By (\cite[Lemma
5.20]{GBF25}),  
this orientation of $Y'$ determines a unique orientation of $Y$.

A {\it stratified homotopy equivalence}
is a homotopy equivalence $f:X\to Y$ with homotopy inverse $g:Y\to X$ such that 
$f$, $g$, and the respective homotopies  from $fg$ to $\text{id}_Y$ and from 
$gf$ to $\text{id}_X$ all satisfy the condition that the image of each 
stratum of the domain is contained in a single stratum of the codomain with the 
same codimension. See \cite[Appendix A]{GBF25} for more details. 
By \cite[Corollary 
5.17]{GBF25}, if $f:X\to Y$ is a stratified homotopy
equivalence then the orientation of $X$ determines an orientation of $Y$.

\begin{theorem}
\label{j16}
Let $f: X\to B\pi$ be a global $F$-Witt map and let $Y$ be a compact stratified PL
pseudomanifold with no codimension one strata..
\begin{enumerate}
\item
Let $g:Y\to X$ be a (topological) homeomorphism 
and
give $Y$ the induced orientation. Then $f\circ g$ is a global $F$-Witt
map and 
$\sW(f\circ g)=\sW(f)$.
\item
Let $g:Y\to X$ be a stratified homotopy equivalence 
and give $Y$ the induced orientation.
Then $f\circ g$ is a global $F$-Witt map and 
$\sW(f\circ g)=\sW(f)$.
\end{enumerate}
\end{theorem}

\begin{proof}
Part 1. To see that $f\circ g$ is a global $F$-Witt map, note that 
$\tilde{g}$ is
a homeomorphism $\tilde{Y}\to\tilde{X}$ and use the fact that a homeomorphism
induces an isomorphism of intersection homology for all classical perversities
(\cite[Theorem 9]{Ki}).

To complete the proof we use more results from \cite{Ki}.  Let $Z$ denote the
intrinsic coarsest stratification of $X$ (\cite[page 150]{Ki}); note that $Z$
is a CS set and not necessarily a stratified pseudomanifold.   Now $\td{Z}$ is
the intrinsic coarsest stratification of $\td{X}$, so the map
$I^{\bar p}H_*(\td{X};F)\to
I^{\bar p}H_*(\td{Z};F)$ is an isomorphism for any classical perversity by
\cite[Theorem 9]{Ki}.  
The proof of Proposition
\ref{f1}.1 shows that $I^{\bar p}C_*(\td{Z};F)$ is chain homotopy equivalent
over $F[\pi]$ to a nonnegatively-graded chain complex of free $F[\pi]$ 
modules, so by \cite[Exercise IV.4.2]{HS} the map
$I^{\bar p}C_*(\td{X};F)\to
I^{\bar p}C_*(\td{Z};F)$
is a chain homotopy equivalence over $F[\pi]$.  In particular, $I^{\bar
p}C_*(\td{Z};F)$ is homotopy finite over $F[\pi]$.

Next define the fundamental class $\Gamma_Z$ of $Z$ to be the image in $I^{\bar 
0}H_*(Z;F)$ of $\Gamma_X$.
It is shown in \cite{GregBook} that the K\"unneth
theorem of \cite{GBF20} generalizes to CS sets, so we can use the method of
Section \ref{j10} to construct a symmetric Poincar\'e complex
\[
(I^{\bar n}C_*(\td{Z};F),\omega)
\]
(note that the conclusion of Theorem {T: universal duality} holds for $Z$
because it holds for $X$).
Let $(I^{\bar n}C_*(\td{X};F),\psi)$ be the representative of $\sW(f)$
constructed in Section \ref{j10}.
The map $I^{\bar n}C_*(\td{X};F)\to I^{\bar n}C_*(\td{Z};F)$ gives a
homotopy equivalence of symmetric Poincar\'e complexes
(\cite[Definition 1.6(ii)]{Ra92})
\[
(I^{\bar n}C_*(\td{X};F),\psi)
\to 
(I^{\bar n}C_*(\td{Z};F),\omega),
\]
and therefore (\cite[Proposition 1.13]{Ra92}) 
$(I^{\bar n}C_*(\td{Z};F),\omega)$ represents $\sW(f)$.

The homeomorphism $g$ determines a stratification of the underlying space of
$X$, which we will denote by $X'$.  It is clear that $\sW(f\circ g)$ is the
same as $\sW(f:X'\to B\pi)$.  The argument given above shows that 
$(I^{\bar n}C_*(\td{Z};F),\omega)$ represents $\sW(f:X'\to B\pi)$ (and hence
that $\sW(f\circ g)=\sW(f)$ as required) provided that
\begin{equation}
\label{r5}
i_*\Gamma_{X'}=\Gamma_Z,
\end{equation}
where $i$ is the map $X'\to Z$.
To verify \eqref{r5}, let $U$ be a Euclidean neighborhood  which is in the 
nonsingular strata of both $X$ and $X'$ (and hence also of $Z$). Let $x\in 
U$ and consider the following diagram: 
\[
\xymatrix{
I^{\bar 0}H_*(X';F)
\ar[r]^-{i_*}
\ar[d]
&
I^{\bar 0}H_*(Z;F)
\ar[d]
&
I^{\bar 0}H_*(X;F)
\ar[l]
\ar[d]
\\
I^{\bar 0}H_*(X',X'-\{x\};F)
\ar[r]^-{i_*}
&
I^{\bar 0}H_*(Z,Z-\{x\};F)
&
I^{\bar 0}H_*(X,X-\{x\};F)
\ar[l]
\\
I^{\bar 0}H_*(U,U-\{x\};F)
\ar[r]^=
\ar[u]_\cong
&
I^{\bar 0}H_*(U,U-\{x\};F)
\ar[u]_\cong
&
I^{\bar 0}H_*(U,U-\{x\};F)
\ar[l]_=
\ar[u]_\cong
}
\]
By \cite[Corollary 5.16]{GBF25}, to show that
$(i^*)^{-1}(\Gamma_Z)=\Gamma_{X'}$, it suffices to 
suffices to show that the image of $(i^*)^{-1}(\Gamma_Z)$ under the left-hand
vertical composite is the orientation class, and this follows from
commutativity of the diagram and the definition of $\Gamma_Z$.

Part 2. 
To see that $f\circ g$ is a global $F$-Witt map, note that $\tilde{g}$ is
a stratified homotopy equivalence $\tilde{Y}\to\tilde{X}$ and use the fact 
that a stratified homotopy equivalence
induces an isomorphism of intersection homology for all perversities
(\cite[Appendix A]{GBF25}).
Now let $(I^{\bar n}C_*(\td{Y};F),\psi)$ be a representative for
$\sW(f\circ g)$.  
By \cite[Proposition 2.1]{GBF3}, $g_*$
gives a homotopy equivalence of symmetric Poincar\'e complexes
from
$(I^{\bar n}C_*(\td{Y};F),\psi)$
to
$(I^{\bar n}C_*(\td{X};F),(g_*\otimes g_*)\psi)$, so
it suffices to show that
$(I^{\bar n}C_*(\td{X};F),(g_*\otimes g_*)\psi)$
is a representative for $\sW(f)$. 
For this in turn it suffices to show that the composite
\begin{multline*}
W\xrightarrow{\psi}
I^{\bar n}C_*(\td{Y};F)^t\otimes_{F[\pi]}
I^{\bar n}C_*(\td{Y};F)
\cong
F\otimes_{F[\pi]}
(I^{\bar n}C_*(\td{Y};F)\otimes_F I^{\bar n}C_*(\td{Y};F))
\\
\xrightarrow{1\otimes \times}
F\otimes_{F[\pi]}
(I^{Q_{\bar n ,\bar n}}C_*(\td{Y}\times \td{Y};F))
\xrightarrow{(g\times g)_*}
F\otimes_{F[\pi]}
(I^{Q_{\bar n ,\bar n}}C_*(\td{X}\times \td{X};F))
\end{multline*}
is a representative for the class $c_f$ defined in Definition
\ref{j17}(ii).
By the definition of $\sW(f\circ g)$ and naturality of the cross product,  
this composite is the same as
\begin{multline*}
W\to \Z\xrightarrow{\zeta_Y}
F\otimes_{F[\pi]} I^{\bar 0}C_*(\td{Y};F)
\xrightarrow{1\otimes g_*}
F\otimes_{F[\pi]} I^{\bar 0}C_*(\td{X};F)
\xrightarrow{1\otimes d}
F\otimes_{F[\pi]} I^{Q_{\bar n ,\bar n}}C_*(\td{X}\times \td{X};F),
\end{multline*}
and now it suffices to observe that, because $Y$ was given the induced
orientation, the 
composite 
\[
\Z\xrightarrow{\zeta_Y}
F\otimes_{F[\pi]} I^{\bar 0}C_*(\td{Y};F)
\xrightarrow{1\otimes g_*}
F\otimes_{F[\pi]} I^{\bar 0}C_*(\td{X};F)
\]
is a representative for the class $b_X$ of Definition \ref{j17}(i).
\end{proof}

Next we prove invariance of $\sW$ under bordism of global $F$-Witt maps.

\begin{theorem}
\label{j18}
Let $f:X\to B\pi$ be a global $F$-$\partial$-Witt map.
Let $Y$ be the boundary of $X$ with the induced orientation. 
Then $\sW(f|_Y)=0$.
\end{theorem}

\begin{proof}
The idea of the proof is to use the method of Section \ref{j10} to construct a
symmetric Poincar\'e pair (Definition \ref{r1}) from the pair $(X,Y)$. 

It's convenient to introduce some notation: given chain complexes $C$ and $D$
and a chain map $f:C\to D$, we write $H_*(D,C)$ for the homology of the mapping
cone $Cf$ (this agrees with the usual meaning of $H_*(D,C)$ when $f$ is a
monomorphism).  An element of $H_*(D,C)$ is represented by a pair $(d,c)$ with 
$\bd c=0$ and $\bd d=-f(c)$.

By Proposition \ref{n3}.3 and the five lemma, the map
\[
H_n(F\otimes_{F[\pi]} I^{\bar 0}C_*(\td{X}),
F\otimes_{F[\pi]} I^{\bar 0}C_*(\td{Y}))
\to
I^{\bar 0}H_n(X,Y;F)
\]
is an isomorphism.
Let $b$ map to the fundamental class of $X$ (\cite[Section 7.2]{GBF25}) under 
this isomorphism, and let $(\eta,\theta)$ be a cycle 
representing $b$. Let 
\[
c\in H_n(\Hom_{\Z[\Z/2]}(W,F\otimes_{F[\pi]} I^{\bar 0}C_*(\td{X}\times \td{X})),
\Hom_{\Z[\Z/2]}(W,F\otimes_{F[\pi]} I^{\bar 0}C_*(\td{Y}\times \td{Y})))
\]
be the class represented by the pair of maps
\[
W\xrightarrow{\varepsilon}
\Z\xrightarrow{\eta}
F\otimes_{F[\pi]} I^{\bar 0}C_*(\td{X})
\xrightarrow{1\otimes d}
F\otimes_{F[\pi]} I^{Q_{\bar n ,\bar n}}C_*(\td{X}\times \td{X})
\]
and
\[
W\xrightarrow{\varepsilon}
\Z\xrightarrow{\theta}
F\otimes_{F[\pi]} I^{\bar 0}C_*(\td{Y})
\xrightarrow{1\otimes d}
F\otimes_{F[\pi]} I^{Q_{\bar n ,\bar n}}C_*(\td{Y}\times \td{Y}).
\]
As in Section \ref{j10}, there is an isomorphism
\begin{multline*}
\Upsilon:
H_*(\Hom_{\Z[\Z/2]}(W,(I^{\bar n}C_*(\td X;F))^t\otimes_{F[\pi]} I^{\bar
n}C_*(\td X;F)),
\\
\Hom_{\Z[\Z/2]}(W,(I^{\bar n}C_*(\td Y;F))^t\otimes_{F[\pi]} I^{\bar
n}C_*(\td
Y;F)))
\\
\xrightarrow{\cong}
H_*(\Hom_{\Z[\Z/2]}(W,F\otimes_{F[\pi]} I^{Q_{\bar n ,\bar n}}C_*(\td X\times
\td X;F)),
\\
\Hom_{\Z[\Z/2]}(W,F\otimes_{F[\pi]} I^{Q_{\bar n ,\bar n}}C_*(\td Y\times
\td Y;F))).
\end{multline*}
Let 
\[
\psi:W\to (I^{\bar n}C_*(\td X;F))^t\otimes_{F[\pi]} I^{\bar
n}C_*(\td X;F)
\]
be a $\Z/2$-equivariant chain map that represents $\Upsilon^{-1}(c)$.  The
proof of Proposition \ref{j6} adapts (using Theorem \ref{T: univ lef}) to show that 
\[
((I^{\bar n}C_*(\td X;F),\psi),(I^{\bar n}C_*(\td Y;F),\partial\psi))
\]
is a symmetric Poincar\'e pair, and since $(I^{\bar n}C_*(\td 
Y;F),\partial\psi)$ is a representative for $\sW(f|_Y)$ (by \cite[Proposition
7.9]{GBF25}), we see that 
$\sW(f|_Y)=0$.
\end{proof}

Finally, recall from \cite[Theorem 1.2]{ALMP-combo} the signature index class 
\[
{\mathrm{Ind}}(\tilde{\eth}_{\mathrm{sign}})\in K_*(C_f^*\pi) 
\]
associated to a smoothly stratified $\mathbb Q$-Witt space $X$ with a map $f:X\to B\pi$.  Also
recall from \cite[Section 11.2]{ALMP-combo} the map
\[
\nu\beta_{\mathbb{Q}}:L^*({\mathbb{Q}}\pi)\to K_*(C_f^*\pi).
\]

\begin{proposition}
${\mathrm{Ind}}(\tilde{\eth}_{\mathrm{sign}})$ 
and 
$\nu\beta_{\mathbb{Q}}(\sW(f))$ 
are equal
in $K_*(C_f^*\pi)\otimes \mathbb Q$.
\end{proposition}

\begin{proof}
The proof is similar to the proof of \cite[Proposition 11.1]{ALMP-combo}.
Let $\Omega^{Witt,s}_*(B\pi)$ be the bordism group over $B\pi$ of
smoothly stratified oriented Witt spaces \cite[see Sections 2.1 and 7]{ALMP-combo}.
Proposition \ref{j14} and Theorem 
\ref{j18} show that $\sW$ is a homomorphism from $\Omega^{Witt,s}_*(B\pi)$ to
$L^*({\mathbb{Q}}\pi)$, and hence we obtain a homomorphism
\[
\nu\beta_{\mathrm{Q}}\sW\otimes \mathbb Q:
\Omega^{Witt,s}_*(B\pi)\otimes \mathbb Q\to K_*(C_f^*\pi)\otimes \mathbb Q.
\]
The natural map $\Omega^{SO}_*(B\pi)\to 
\Omega^{Witt,s}_*(B\pi)$ is a rational surjection by  \cite[Proposition 
7.3]{ALMP-combo} (see also \cite{Cu92, BCS} for the analogous result for PL Witt bordism),
and so 
it suffices to check 
that ${\mathrm{Ind}}(\tilde{\eth}_{\mathrm{sign}})$ 
and 
$\nu\beta_{\mathrm{Q}}(\sW(f))$ agree when $X$ is a smooth manifold. This in
turn follows from Proposition \ref{j12} and a result of Kasparov and 
Mi\v{s}\v{c}enko \cite{Ka95} 
that identifies the rational symmetric signature and the signature index 
class for smooth closed manifolds. 
\end{proof}

\section{Technical facts about intersection chains}\label{S: covers}

In this section, we prove some results that were needed in previous 
sections and in \cite{GBF25}.  

Throughout this section we fix an $n$-dimensional 
$\bd$-stratified 
pseudomanifold $X$ and a regular cover $p:\td X\to X$. We write $\pi$ for the 
group of covering translations.  For any subset $S$ of $X$ we write $\td{S}$
for $p^{-1}(S)$.  

Recall that an open set $U$ in $X$ is called {\it evenly covered} if the
restriction of the covering map $p$ to $U$ is trivial.

We also fix a perversity $\bar p$.

\subsection{A colimit formula for intersection chains}\label{S: 6.1}

Let $\U$ be a covering of $X$ by open sets. 
Let $\cC$ be the category of all finite intersections of sets in
$\U$, with inclusions as the morphisms. 
Let $A$ be an open subset of $X$.
Fix a ring $R$ and an $R$-module $M$.

Our main result in this subsection is

\begin{proposition}
\label{n3}
\begin{enumerate}
\item
The canonical map
\[
\dlim_{V\in \cC}\, I^{\bar p}C_* (V,V\cap A;M)\to I^{\bar p}C_* (X,A;M)
\]
is a chain homotopy equivalence \red{over $R$} .

\item
The canonical map
\[
\dlim_{V\in \cC}\, I^{\bar p}C_* (\td{V},\td{V}\cap\td{A};M)\to 
I^{\bar p}C_* (\td{X},\td{A};M)
\]
is a chain homotopy equivalence over $R[\pi]$.

\item
The projection
\[
R\otimes_{R[\pi]} I^{\bar p}C_*(\td{X},\td A;M)
\to
I^{\bar p}C_*(X,A;M)
\]
is a chain homotopy equivalence over $R$.

\end{enumerate}
\end{proposition}

\begin{remark}
It's possible that the projection in part 3 is actually an isomorphism, as
it is for ordinary singular chains.
\end{remark}

For the proof of Proposition \ref{n3} we need a preliminary result which may 
be of interest in its own right.
Let $I^{\bar p}_\U C_*(X,A;M)$ denote the submodule
\[
\sum_{U\in \U}\, I^{\bar p}C_*(U,U\cap A;M)
\]
of $I^{\bar p}C_*(X,A;M)$.

\begin{proposition}
\label{n1}
The canonical map
\[
\dlim_{V\in \cC}\, I^{\bar p}C_* (V,V\cap A;M)\to 
I^{\bar p}_\U C_*(X,A;M)
\]
is an isomorphism.
\end{proposition}

The proof of Proposition \ref{n1} will be given in Subsection \ref{n14}.

\begin{proof}[Proof of Proposition \ref{n3}.]
Proposition 2.9 of \cite{GBF10} states that the inclusion
\[
I^{\bar p}_\U C_*(X;M)
\hookrightarrow
I^{\bar p} C_*(X;M)
\]
is a chain homotopy equivalence.
The same proof shows that the inclusion
\[
I^{\bar p}_\U C_*(X,A;M)
\hookrightarrow
I^{\bar p} C_*(X,A;M)
\]
is a chain homotopy equivalence, and part 1 follows from this and Proposition
\ref{n1}.

A minor modification of the proof in \cite{GBF10} shows that the inclusion
\[
I^{\bar p}_\U C_* (\td{X},\td{A};M)
\hookrightarrow
I^{\bar p}C_* (\td{X},\td{A};M)
\]
is a chain homotopy equivalence over $R[\pi]$, and part 2 follows from this
and Proposition \ref{n1} (applied to the pair $(\td X,\td A))$.

For part 3 we assume that the open sets in $\U$ are evenly covered.
With this assumption the projection
\begin{multline*}
R\otimes_{R[\pi]} \left(
\dlim_{V\in \cC}\, I^{\bar p}C_* (\td{V},\td{V}\cap\td{A};M)
\right)
=
\dlim_{V\in \cC}\, R\otimes_{R[\pi]}I^{\bar p}C_* (\td{V},\td{V}\cap\td{A};M)
\\
\to
\dlim_{V\in \cC}\, I^{\bar p}C_* (V,V\cap A;M)
\end{multline*}
is an isomorphism.  Now consider the diagram
\[
\xymatrix{
R\otimes_{R[\pi]} \bigl(
\dlim_{V\in \cC}\, I^{\bar p}C_* (\td{V},\td{V}\cap\td{A};M)
\bigr)
\ar[r]^-\cong
\ar[d]
&
\dlim_{V\in \cC}\, I^{\bar p}C_* (V,V\cap A;M)
\ar[d]
\\
R\otimes_{R[\pi]} I^{\bar p}C_*(\td X,\td A;M)
\ar[r]
&
I^{\bar p}C_*(X,A;M)
}
\]
The right vertical arrow is a chain homotopy equivalence over $R$ by part 1.  
The left vertical arrow is a chain homotopy equivalence over $R$ by part 2.
Hence the lower horizontal arrow is a chain homotopy equivalence over $R$ as 
required.
\end{proof}

\subsection{Freeness and flatness}

In this section we prove two results.  Let $A$ be an open subset
of $X$ and let $F$ be a field.

First we have

\begin{proposition}
\label{f1}
\begin{enumerate}
\item 
If $X$ has a finite covering by evenly covered open sets (in particular, if 
$X$ is compact) then $I^{\bar p}C_*(\td{X},\td{A};F)$ is chain 
homotopy equivalent over $F[\pi]$ to a nonnegatively-graded chain complex of 
free $F[\pi]$-modules.
\item
For all $X$, $I^{\bar p}C_*(\td{X},\td{A};F)$ is chain homotopy
equivalent over $F[\pi]$ to a nonnegatively-graded chain complex of flat 
$F[\pi]$-modules.
\end{enumerate}
\end{proposition}

For the second result we let $\pi$ act by the diagonal action on
$I^{\bar p}C_*(\td X,\td A;F)\otimes_F I^{\bar q}C_*(\td X,\td A;F)$ and on 
$I^{Q_{\bar p,\bar q}}C_*(\td X\times \td X,\td A\times \td X\cup \td X\times
\td A;F)$. 

\begin{proposition}\label{P: tensor flat}
The cross product 
\[
I^{\bar p}C_*(\td X,\td A;F)\otimes_F I^{\bar q}C_*(\td X,\td A;F)
\to
I^{Q_{\bar p,\bar q}}C_*(\td X\times \td X,\td A\times \td X\cup \td X\times
\td A;F)
\]
induces a quasi-isomorphism
\begin{multline*}
F\otimes_{F[\pi]} 
(
I^{\bar p}C_*(\td X,\td A;F)\otimes_F I^{\bar q}C_*(\td X,\td A;F)
)
\to 
F\otimes_{F[\pi]} I^{Q_{\bar p,\bar q}}C_*(\td X\times \td X,\td A\times \td X\cup \td X\times
\td A;F).
\end{multline*}
\end{proposition}

For both results we will give the proofs when $A=\emptyset$; the same proofs
work for the general cases. Note that although only an absolute version version of the K\"unneth theorem is provided in \cite{GBF20}, it is extended to a relative version in \cite[Appendix B]{GBF25}. 

For the proof of Proposition \ref{f1} we need a lemma.

\begin{lemma}
\label{n15}
Let $\V$ be a finite collection of evenly covered open sets in $X$.
Let $\cD$ be the 
category of intersections of sets of
$\V$, with inclusions as the morphisms.  Then
\[
\dlim_{V\in \cD}\, I^{\bar p}C_* (\td V;F)
\]
is free over $F[\pi]$.
\end{lemma}

\begin{proof}
For each $V$ in $\cD$ let $A(V)$ be the image of the map
\[
\dlim I^{\bar p}C_* (\td{W};F)\to I^{\bar p}C_* (\td{V};F),
\]
where the colimit is taken over $W\in \cD$ with $W\subsetneq V$, and let 
$B(V)$ be the cokernel of $A(V)\to I^{\bar p}C_* (\td{V};F)$.
$B(V)$ is free over $F[\pi]$ because $V$ and all $W$ are evenly covered
and $F$ is a field.

Let $G$ be the functor on $\cD$  which takes $V$ to $I^{\bar p}C_* (\td 
V;F)$, and let $H$ be the functor which takes $V$ to $\oplus_{W\subset V}
B(W)$.  We claim that $G$ is naturally isomorphic to $H$, and hence that
\[
\dlim_{V\in \cD}\, I^{\bar p}C_* (\td V;F)
\cong
\bigoplus_{V\in \cD} B(V),
\]
which immediately implies the result.

To verify the claim, first note that, because $B(V)$ is free over $F[\pi]$,
the short exact sequence
\[
0\to A(V)\to I^{\bar p}C_* (\td{V};F)\to B(V)\to 0
\]
of $F[\pi]$-modules is split for all $V$.  Now define the {\it complexity} of 
$V$ to be the 
number of objects of $\cD$ which are contained in $V$, and let $\cD_i$ be the 
full subcategory of $\cD$ with objects of complexity $\leq i$.  It is easy to
see by induction on $i$ that the restrictions of $G$ and $H$ to $\cD_i$ are
naturally isomorphic, which in particular gives the claim.
\end{proof}

\begin{proof}[Proof of Proposition \ref{f1}]
Part 1 is immediate from Lemma \ref{n15} and Proposition \ref{n3}.2.

For part 2, let $\U$ be a collection of evenly covered open sets whose union is
$X$, and let $\cC$ be the
category of finite intersections of sets in $\U$.  For each finite
subset $\V$ of $\U$
let $\cD(\V)$ be the category of intersections of sets in $\V$.  Then
\[
\dlim_{V\in \cC}\, I^{\bar p}C_* (\td{V};F)
=\dlim_\V \,\dlim_{V\in \cD{(\V)}}\, I^{\bar p}C_* (\td{V};F)
\]
and the result follows from Proposition \ref{n3}.2, Lemma \ref{n15}, and the 
fact that a directed colimit of flat modules is flat.
\end{proof}

\begin{proof}[Proof of Proposition \ref{P: tensor flat}]
Let $C$ and $D$ denote
$I^{\bar p}C_*(\td X;F)\otimes_F I^{\bar q}C_*(\td X;F)$ 
and
$I^{Q_{\bar p,\bar q}}C_*(\td X\times \td X;F)$
respectively.  
Let $R$ denote $F[\pi\times\pi]$, which is isomorphic to $F[\pi]\otimes
F[\pi]$. 
By Proposition \ref{f1}.2, we have chain homotopy equivalences
$C\to C'$ and $D\to D'$ 
over $R$ (and hence
over $F[\pi]$), where $C'$ and $D'$ are nonnegatively-graded and flat over 
$R$.  But $R$ is flat (in fact free) over $F[\pi]$, and hence $C'$ and $D'$ 
are flat over
$F[\pi]$ (because the functor $C'\otimes_{F[\pi]} \,-$ is naturally 
isomorphic to $C'\otimes_R\, R\otimes_{F[\pi]}\,-$, and similarly for $D'$). 
Now the map
\[
C\xrightarrow{\times} D
\]
is a quasi-isomorphism by the K\"unneth theorem of \cite{GBF20}, and hence the
composite
\[
C'\to C\to D\to D'
\]
induces a quasi-isomorphism 
\[
F\otimes_{F[\pi]} C'\to F\otimes_{F[\pi]} D'
\]
by \cite[Theorem 5.6.4]{WEIB}.  The maps
$F\otimes_{F[\pi]} C'\to 
F\otimes_{F[\pi]} C$ and $F\otimes_{F[\pi]} D\to
F\otimes_{F[\pi]} D'$ are quasi-isomorphisms because $F\otimes_{F[\pi]}$
preserves chain homotopy equivalences over $F[\pi]$, so we conclude that
$F\otimes_{F[\pi]} C\to
F\otimes_{F[\pi]} D$ is a quasi-isomorphism as required.
\end{proof}

\subsection{Proof of Proposition \ref{j7}}
\label{j9}
Recall (for example from \cite[Exercise IV.4.2]{HS}) that if two 
bounded-below chain complexes are free over $F[\pi]$ and quasi-isomorphic 
over $F[\pi]$ then they are chain homotopy equivalent over $F[\pi]$.
Combining this with
Proposition \ref{f1}.1, it suffices to show that $I^{\bar p}C_*(\td X;F)$ is
quasi-isomorphic over $F[\pi]$ to a finite $F[\pi]$ chain complex.

This in turn is immediate from the following lemma.
Let $X'$ denote $X-\bd X$.  

\begin{lemma}
\label{n17}
\begin{enumerate}
\item
The map
\[
I^{\bar p}H_*(\widetilde{X'};F)
\to
I^{\bar p}H_*(\td X;F)
\]
induced by the inclusion is an isomorphism.
\item
$I^{\bar p}C_*(\widetilde{X'};F)$ is quasi-isomorphic over $F[\pi]$ to a 
finite $F[\pi]$ chain complex.
\end{enumerate}
\end{lemma}

\begin{remark}
The reason that $X'$ plays a special role is that we will need to use the
relation between intersection homology and the Deligne sheaf, and this relation
is not known for $\bd$-stratified pseudomanifolds with nonempty boundary.
\end{remark}

Before continuing we need to recall some definitions.  Let $K$ be a simplicial
complex.  A subcomplex $L$ of $K$ is {\it full} if every simplex whose 
vertices are in $L$ is in $L$.  Let $s$ be a simplex of $K$.
The {\it closed star} of
$s$ is the union of all the simplices containing it; this will
be denoted $\bSt(s)$.  The {\it open star} of $s$ is the interior of
$\bSt(s)$; this will be denoted $\St(s)$.  

Fix a triangulation of
$X$ with the property that each skeleton of $X$ is a full subcomplex.

For the proof of Lemma \ref{n17} we need two other lemmas, whose proofs we
defer for a moment.

\begin{lemma}
\label{n18}
Let $s$ be a simplex of $X$ which is contained in $X'$. 
Then 
$I^{\bar p}C_*(\widetilde{\St(s)};F)$ is homotopy finite
over $F[\pi]$.
\end{lemma}

\begin{lemma}
\label{n19}
The homotopy pushout (double mapping cylinder) of homotopy finite chain
complexes over $F[\pi]$ is quasi-isomorphic over $F[\pi]$ to a finite 
$F[\pi]$ chain complex. 
\end{lemma}

\begin{proof}[Proof of Lemma \ref{n17}]
Part 1.  By the definition of $\bd$-stratified pseudomanifold (\cite[Definition
7.1]{GBF25}) $\bd X$ has an open collar neighborhood in $X$.  This implies that
the inclusion $\widetilde{X'}\to\td X$ is a stratified
homotopy equivalence, and the result follows from 
\cite[Appendix A]{GBF25}.

Part 2.
First observe that $X'$ is the union of the open stars of the vertices of $X$
that are contained in $X'$ and that there are finitely many such vertices
(because $X$ is compact).  We will also use the fact that the intersection of 
the open stars of finitely many vertices, if it is nonempty, is the open star 
of the simplex determined by these vertices.

We will prove by induction on $k$ that if $U_1,\ldots,U_k$ are open stars
of simplices contained in $X'$ and $U$ is $U_1\cup\cdots\cup U_k$ then 
$I^{\bar p}C_*(\td U;F)$ is quasi-isomorphic over $F[\pi]$ to a
finite $F[\pi]$ chain complex.
Let $V=U_1\cup \cdots\cup U_{k-1}$ and let $W=V\cap U_k$.
Let $C$ be the pushout of the diagram
\[
\tag{*}
\xymatrix{
I^{\bar p}C_*(\td W;F)
\ar[r]
\ar[d]
&
I^{\bar p}C_*(\widetilde{U_k};F)
\\
I^{\bar p}C_*(\td V;F)
}
\]
and let $D$ be its homotopy pushout.
$I^{\bar p}C_*(\td U;F)$ is chain homotopy equivalent to $C$ by Proposition
\ref{n3}.2.
The three chain complexes in diagram (*) are homotopy finite over $F[\pi]$ 
(this follows from the inductive hypothesis, Proposition \ref{f1}.1, and
Lemma \ref{n18}) so by Lemma \ref{n19} $D$ is quasi-isomorphic
over $F[\pi]$ to a finite $F[\pi]$ chain complex.
To conclude the proof we show that the quotient map $D\to C$ is a
quasi-isomorphism.  Diagram (*) gives a Mayer-Vietoris sequence
\[
\cdots\to
I^{\bar p}H_i(\td W;F)
\to
I^{\bar p}H_i(\widetilde{U_k};F) \oplus
I^{\bar p}H_i(\td V;F)
\to
H_i(D)
\to
I^{\bar p}H_{i-1}(\td W;F)
\to
\cdots
\]
There is also a Mayer-Vietoris sequence for $C$ (because the map
\[
I^{\bar p}C_i(\td W;F)
\to
I^{\bar p}C_i(\widetilde{U_k};F) \oplus
I^{\bar p}C_i(\td V;F)
\]
is a monomorphism) so the five lemma shows that $H_*(D)\to H_*(C)$ is an
isomorphism.
\end{proof}

For the proof of Lemma \ref{n18} we need a definition.
The {\it combinatorial link} of $s$, denoted $\Lk(s)$, is the union 
of the simplices of $\bSt(s)$ that do not intersect $s$.

\begin{proof}[Proof of Lemma \ref{n18}]
First recall 
(for example from \cite[Lemma 62.6]{MUNK}) that $\bSt(s)$
is equal to the join $s*\Lk(s)$.

In particular, $\bSt(s)$ is contractible, so 
the covering map $p:\td X\to X$ is 
trivial over
$\St(s)$, and hence
\[
I^{\bar p}C_*(\widetilde{\St(s)};F)
\cong
F[\pi]\otimes I^{\bar p}C_*(\St(s);F).
\]
Thus it suffices to show that $I^{\bar p}C_*(\St(s);F)$ is homotopy finite
over $F$.  But (using the fact that $F$ is a field) $I^{\bar
p}C_*(\St(s);F)$ is chain homotopy equivalent to $I^{\bar
p}H_*(\St(s);F)$, so it suffices to show that the latter is finitely
generated.

Now $s=\hat s*\bd s$, where
$\hat s$ is the barycenter of $s$, and so
$\bSt(s)=\hat s*\bd s*\Lk(s)$. This is homeomorphic to 
the cone on $\bd s*\Lk(s)$, and the homeomorphism takes $\St(s)$
to the open cone
\[
([0,1)\times (\bd s*\Lk(s))/(0\times x\sim 0\times
y)
\]
which we denote by $Q$.  We give $Q$ the stratification determined by the
homeomorphism. Each subspace $(0,1)\times z$ of Q is taken by the inverse
homeomorphism to the interior of a simplex of $X$, and the interior of each
simplex of $X$ is contained in a single stratum, so each subspace $(0,1)\times
z$ is contained in a single stratum of $Q$.  It follows that the subspace
\[
([0,1/2)\times (\bd s*\Lk(s))/(0\times x\sim 0\times y),
\]
which we denote by $P$, is stratified homotopy equivalent to
$Q$ (as defined in \cite[Appendix A]{GBF25}).
Next we recall that $I^{\bar p}H_*$ of an open set in $X'$ is the
hypercohomology of the Deligne sheaf (for general perversities this is
\cite[Theorem 3.6]{GBF23}) and that the Deligne sheaf is cohomologically
constructible (\cite[Proposition 4.1]{GBF23}), which in particular means that
it satisfies Wilder's Property (P,Q) (\cite[page 69]{Bo}).  In our situation
this says that the image of the map $I^{\bar p}H_*(P;F)\to I^{\bar p}H_*(Q;F)$
is finitely generated.  But this map is an isomorphism by 
\cite[Appendix A]{GBF25},
so $I^{\bar p}H_*(Q;F)$ is finitely generated as required.
\end{proof}

\begin{proof}[Proof of Lemma \ref{n19}]
Let 
\[
\tag{*}
\xymatrix{
A
\ar[r]^-g
\ar[d]_f
&
C
\\
B
&
}
\]
be a diagram of homotopy finite chain complexes over $F[\pi]$ and $F[\pi]$ chain
maps.  Recall that the homotopy pushout of diagram $(*)$ is defined as
follows.
Let $I$ denote the cellular chain complex of the unit interval, that is, the 
$F$ chain complex with two generators $a$ and
$b$ in dimension 0, one generator $c$ in dimension 1, and differential
$\bd c=b-a$.   Let $F$ be the chain complex consisting of $F$ in
dimension 0, and let $\alpha,\beta:F\to I$ be the maps which take 1 to $a$ and
$b$ respectively.
Define $B'$ by the pushout diagram
\[
\xymatrix{
A
\ar[d]_f
\ar[r]^-{A\otimes \beta}
&
A\otimes I
\ar[d]
\\
B
\ar[r]
&
B'
}
\]
and similarly for $C'$.
Then the homotopy pushout of $(*)$, which we will denote by $D$,  is defined by
the pushout diagram
\[
\xymatrix{
A
\ar[r]
\ar[d]
&
C'
\ar[d]
\\
B'
\ar[r]
&
D,
}
\]
where the upper horizontal and leftmost vertical arrows are induced by
$A\otimes \alpha$.

Next let $i:A\to \bar{A}$, $j:B\to \bar{B}$, $k:C\to \bar{C}$ be chain homotopy equivalences 
over $F[\pi]$ with $\bar{A}$, $\bar{B}$, $\bar{C}$ finite $F[\pi]$ chain complexes.
Then there are maps $\bar{f}:\bar{A}\to \bar{B}$ and 
and $\bar{g}:\bar{A}\to \bar{C}$ making the diagram
\[
\tag{**}
\xymatrix{
B
\ar[d]_j
&
A
\ar[l]_-f
\ar[r]^-g
\ar[d]_i
&
C
\ar[d]_k
\\
\bar{B}
&
\bar{A}
\ar[l]_-{\bar{f}}
\ar[r]^-{\bar{g}}
&
\bar{C}
}
\]
commute up to chain homotopy.  The homotopy pushout $\bar{D}$ of the second 
row is given by a pushout diagram 
\[
\xymatrix{
\bar{A}
\ar[r]
\ar[d]
&
\bar{C}'
\ar[d]
\\
\bar{B}'
\ar[r]
&
\bar{D}.
}
\]

It is easy to check that $\bar{D}$ is a finite 
$F[\pi]$ chain complex.  
To compare $D$ with $\bar{D}$ we introduce an ``extended'' version of $D$.
Define a chain complex $2I$ by the pushout diagram
\[
\tag{***}
\xymatrix{
F
\ar[r]^-{\alpha}
\ar[d]_{\beta}
&
I
\ar[d]^{\delta}
\\
I
\ar[r]^-{\gamma}
&
2I.
}
\]
Let $\zeta$ (resp., $\eta$) be the composite
$F \xrightarrow{\alpha} I\xrightarrow{\gamma} 2I$
(resp., $F \xrightarrow{\beta} I\xrightarrow{\delta} 2I$).
Replacing $I$ by $2I$, $\alpha$ by $\zeta$, and $\beta$ by $\eta$ in the
construction of $D$ gives a pushout diagram
\[
\xymatrix{
A
\ar[r]
\ar[d]
&
C''
\ar[d]
\\
B''
\ar[r]
&
E.
}
\]

Next we construct a quasi-isomorphism $E\to D$.
In diagram (***), write $I_1\subset 2I$ for the image of $\gamma$ and $I_2$ for
the image of $\delta$.  Also let $\epsilon:I\to F$ be the chain map which takes
$a$ and $b$ to 1.
Define a map $\theta:2I\to I$ by letting $\theta$ be
$\gamma^{-1}$ on $I_1$ and $\beta\circ\epsilon\circ\delta^{-1}$ on $I_2$.
$\theta$ induces maps $B''\to B'$ and $C''\to C'$ and hence a map $E\to D$.
Applying the five lemma to the Mayer-Vietoris sequences of $E$ and $D$ shows
that the map $E\to D$ is a quasi-isomorphism.

Finally, we construct a quasi-isomorphism $E\to \bar{D}$.  In diagram (**), let
$H:A\otimes I\to\bar{B}$ be the chain homotopy from $\bar{f}\circ i$ to 
$j\circ f$.  Define a map $\kappa:B''\to \bar{B}'$ to be $A\otimes 
\gamma^{-1}$ on
$A\otimes I_1$, $H\circ (A\otimes \delta^{-1})$ on $A\otimes I_2$, and $j$ on
$B$.  Similarly, define a map $\lambda:C''\to \bar{C}$.  Then $\kappa$ and
$\lambda$ give a map $E\to \bar{D}$, and applying the five lemma to the
Mayer-Vietoris sequences of $E$ and $\bar{D}$ shows that this is a
quasi-isomorphism.
\end{proof}

\subsection{Proof of Proposition \ref{n1}}
\label{n14}

\red{We continue to use the notation from the beginning of Section \ref{S: 
6.1}.}

We need a lemma, whose proof we defer for a moment.

\begin{lemma}
\label{n2}
Let $U_1,\ldots,U_m\in\U$ and let $\xi_i\in I^{\bar
p}C_*(U_i;M)$ for $1\leq i\leq m$
with
\[
\sum \xi_i=0
\]
in $I^{\bar p}C_*(X,A;M)$.  Then for $2\leq i\leq m$ there exist $\eta_i\in 
I^{\bar p}C_*(U_i\cap U_1;M)$ with
\[
\xi_1+\sum_{i=2}^m \eta_i=0
\]
in $I^{\bar p}C_*(X,A;M)$.
\end{lemma}

\begin{proof}[Proof of Proposition \ref{n1}.]
Let $K$ be the kernel 
of the canonical epimorphism
\[
\bigoplus_{U\in\U} I^{\bar p}C_*(U,U\cap A;M)
\to
I^{\bar p}_\U C_*(X,A;M)
\]
and let $L$ be the kernel of the canonical epimorphism
\[
\bigoplus_{U\in\U} I^{\bar p}C_*(U,U\cap A;M)
\to
\dlim_{V\in \cC}\, I^{\bar p}C_* (V,V\cap A;M).
\]
For a chain $\xi\in I^{\bar p}C_*(U;M)$, let $[\xi]$ denote its image in 
$I^{\bar p}C_*(U,U\cap A;M)$.
$K$ is generated by tuples
\[
([\xi_1],\ldots,[\xi_m])\in \bigoplus_{i=1}^m I^{\bar p}C_*(U_i,U_i\cap 
A;M)
\]
with $\sum \xi_i=0$ in $I^{\bar p}C_*(X,A;M)$, as $(U_1,\ldots,U_m)$ ranges
over all $m$-tuples in $\U$. 
$L$ is generated by pairs
\[
([\xi],[-\xi])\in I^{\bar p}C_*(U,U\cap A;M)\oplus I^{\bar p}C_*(U',U'\cap
A;M)  
\]
with $\xi\in I^{\bar p}C_*(U\cap U';M)$, as $(U,U')$
ranges over all pairs in $\U$.  It's clear that $L\subset K$ and it suffices to
show that each of the generating tuples for $K$ is in $L$.  So let
$([\xi_1],\ldots,[\xi_m])$ be such a tuple.  We assume inductively that all 
shorter
such tuples are in $L$.
Lemma \ref{n2} gives an equation
\begin{multline*}
([\xi_1],\ldots,[\xi_m])
=
([-\eta_2],[\eta_2],0,\ldots,0)
+
([-\eta_3],0,[\eta_3],0,\ldots,0)
+\cdots 
+([-\eta_m],0,\ldots,0,[\eta_m])
\\
+(0,[\xi_2-\eta_2],\ldots,[\xi_m-\eta_m]).
\end{multline*}
The last summand on the right is in $L$ by the inductive hypothesis, and the
remaining summands are obviously in $L$.  
\end{proof}

\begin{proof}[Proof of Lemma \ref{n2}.]
We begin with the case $A=\emptyset$.

For a chain $\xi$ and a singular simplex $\sigma$ with the same dimension as
$\xi$, we write 
\[
c_\xi(\sigma)
\]
for the coefficient of $\sigma$ in $\xi$.  
We say that $\sigma$ {\it belongs} to $\xi$ if $c_\xi(\sigma)\neq 0$.

Let $U_i$ and $\xi_i$, $1\leq i\leq m$, be as in the lemma.  For 
$2\leq i\leq m$, let ${\mathfrak A}_i$ be the set of singular simplices which belong to 
both $\xi_i$ and $\xi_1$, and let
\[
\theta_i=\sum_{\sigma\in {\mathfrak A}_i}
c_{\xi_i}(\sigma)\sigma.
\]
The equation $\sum_{i=1}^m \xi_i=0$ implies
\begin{equation}
\label{n4}
\xi_1+\sum_{i=2}^m\theta_i=0,
\end{equation}
which might suggest we could take $\eta_i$ to be $\theta_i$,
but $\theta_i$ will not be an intersection chain in 
general because its boundary can contain non-allowable simplices that cancel
out in $\xi_i$.

For $2\leq i\leq m$, let ${\mathfrak B}_i$ be the set of singular simplices which 
belong to $\xi_i$ and intersect $U_1$ but do not belong to $\xi_1$.
Let ${\mathfrak B}=\cup_{i=2}^m {\mathfrak B}_i$.
The equation $\sum_{i=1}^m \xi_i=0$ implies
\begin{equation}
\label{n5}
\sum_{i=2}^m c_{\xi_i}(\sigma)=0
\end{equation}
for each $\sigma\in {\mathfrak B}$.

The strategy of the rest of the proof is to replace each $\sigma$ in
${\mathfrak B}$ by a
chain $\bar\sigma$, in such a way that for $2\leq i\leq m$

\renewcommand{\theenumi}{\Roman{enumi}}
\renewcommand{\labelenumi}{(\theenumi)}

\begin{enumerate}
\item
\label{n7}
the support $|\bar\sigma|$ is contained in $|\sigma|\cap U_1$, and 
\item
\label{n8}
the chain
$
\theta_i
+
\sum_{\sigma\in {\mathfrak B}_i}
c_{\xi_i}(\sigma)\bar\sigma
$
is allowable. 
\end{enumerate}
We can then let $\eta_i$ be $\theta_i
+
\sum_{\sigma\in {\mathfrak B}_i}
c_{\xi_i}(\sigma)\bar\sigma$; the equation 
$\xi_1+\sum_{i=2}^m\eta_i=0$ will follow from equations \eqref{n4} and 
\eqref{n5}.

We will construct the chains $\bar\sigma$ by using the subdivision procedure in
the proof of \cite[Proposition 2.9]{GBF10} (with the ordered cover $U_1,X$);
for the convenience of the reader we give the details.

First we need some notation.  Suppose we are given 
\begin{itemize}
\item
a singular simplex $\tau:\Delta^j\to X$, 
\item
a simplicial complex 
$K$ which is a subdivision of $\Delta^j$, and
\item
an ordering of the vertices of $K$ which is a total ordering on the 
vertices of each simplex. 
\end{itemize}
For each $j$-dimensional simplex $s$ of $K$ the total ordering of the 
vertices of $s$ determines an affine isomorphism 
\[
\i_s:\Delta^j\to s.
\]
Let $\epsilon_s$ be 1 if the total ordering of the vertices of $s$ agrees with
the orientation inherited from $\Delta^j$ and $-1$ otherwise.  Let 
\begin{equation}
\label{n10}
\i_K=\sum\epsilon_s\,\i_s,
\end{equation}
where the sum is taken over all $j$-dimensional simplices of $K$.  Then $\i_K$
is a singular chain of $\Delta^j$.  The chain $\tau_*(\i_K)$ is the 
subdivision of $\tau$ determined by the given data.

Now suppose in addition that $\tau$ is allowable.  Then \cite[Lemma 
2.6]{GBF10} says that for every $j$-dimensional simplex $s$ of $K$ 
the singular simplex $\tau\circ\i_s$ is allowable.  Also, if $t$ is a 
$(j-1)$-dimensional simplex of $K$ then a straightforward argument (which is written
out on page 1993 of \cite{GBF10}) shows that $\tau\circ\i_t$
is allowable except perhaps when $t$ contains a simplex $u$ 
which is contained in the $\dim(u)$-skeleton of $\Delta^j$.  We will call a
simplex $u$ of $K$ which is contained in the $\dim(u)$-skeleton of $\Delta^j$ 
{\it awkward} (with respect to $\tau$).  

Let $k$ denote the dimension of the chains $\xi_i$.  For $0\leq j\leq k$, let
${\mathfrak B}^j$ denote the set of singular simplices of dimension $j$ which 
are faces of singular simplices in ${\mathfrak B}$ (in particular 
${\mathfrak B}^k={\mathfrak B}$).  By 
induction on $j$, we will construct for each $\tau\in {\mathfrak B}^j$
\begin{itemize}
\item
a subdivision $K_\tau$ of $\Delta^j$, and
\item
a partial ordering of the vertices of $K_\tau$ which restricts to a total
ordering on the vertices of each simplex, 
\end{itemize}
with the following properties.\footnote{Property \eqref{n6} is a slight 
modification of the procedure in \cite{GBF10}.}
\renewcommand{\theenumi}{\roman{enumi}}
\renewcommand{\labelenumi}{(\theenumi)}
\begin{enumerate}
\item
\label{n6}
If $|\tau|\subset U_1$ then $K_\tau=\Delta^j$.
\item
Under the identification of the $l$-th face of $\Delta^j$ with $\Delta^{j-1}$,
the subdivision of the $l$-th face agrees with $K_{\bd_l\tau}$.
\item
\label{n9}
If $u$ is an awkward simplex of $K_\tau$ which
is contained in $\tau^{-1}(U_1)$, then any simplex of $K_\tau$ containing $u$ 
is contained in $\tau^{-1}(U_1)$.
\end{enumerate}
For $j=0$, $K_\tau=\Delta^0$.  Suppose the construction has
been accomplished for all dimensions $<j$ and let $\tau\in {\mathfrak B}^j$ with $|\tau|$
not contained in $U_1$.  The 
subdivisions associated to the faces of $\tau$ give a simplicial complex $K_0$
which is a subdivision of the boundary of $\Delta^j$.  
Let $\Delta'$ be the cone on $K_0$.
Then $K_0$ is a subcomplex of $\Delta'$ so we can apply
barycentric subdivision holding $K_0$ fixed (see \cite[page 89]{MUNK} for the
definition) until Property \eqref{n9} is satisfied (see the proof of 
\cite[Lemma 16.3]{MUNK}).  We order the vertices at each stage of the
subdivision process by letting each new vertex be greater than all the 
existing vertices adjacent to it.

Now for each $\sigma\in {\mathfrak B}^k$ we let
\[
\bar{\sigma}=\sum \epsilon_s\, \sigma\circ\i_s
\]
where the sum is over all simplices $s$ of $K_\sigma$ that are contained in
$\sigma^{-1}(U_1)$.  Also, for each $\tau\in {\mathfrak B}^{k-1}$, we let
\[
\bar{\tau}=\sum \epsilon_t\, \tau\circ\i_t
\]
where the sum is over all simplices $t$ of $K_\tau$ that are contained in
$\tau^{-1}(U_1)$.

We need to show that the $\bar\sigma$ satisfy Properties \eqref{n7} and
\eqref{n8} above.
Property \eqref{n7} is clearly satisfied.
As a first step toward Property \eqref{n8}, we calculate 
$\partial\bar{\sigma}$ modulo allowable singular simplices.  Fix a 
$\sigma\in {\mathfrak B}^k$ and let 
\[
{\mathfrak j}=\sum \epsilon_s \i_s,
\]
where the sum is over all simplices of $K_\sigma$ that  are contained in
$\sigma^{-1}(U_1)$; then $\bar{\sigma}=\sigma_*{\mathfrak j}$ and
$\partial\bar{\sigma}=\sigma_*(\bd\mathfrak j)$.  
Suppose that $t$ is a $(k-1)$-simplex belonging to $\bd{\mathfrak j}$ such 
that $\sigma\circ\i_t$ is non-allowable.  Then $t$ must contain an awkward 
simplex of $K_\sigma$, so Property \eqref{n9} implies that the coefficient of 
$\i_t$ in $\bd\mathfrak j$ is the same as its coefficient in 
$\bd\i_{K_\sigma}$ (see equation \eqref{n10}). If $t$ is not 
contained in $\bd\Delta^k$ then this coefficient is 0.  If $t$ is contained in
the $l$-the face of $\Delta^k$ then (identifying this face with
$\Delta^{k-1}$) this coefficient is $(-1)^l \epsilon_t$.
It follows that
\begin{equation}
\label{n11}
\bd\bar\sigma\equiv\sum_{\tau\in {\mathfrak B}^{k-1}} c_{\partial\sigma}(\tau) \bar{\tau}
\end{equation}
modulo allowable singular simplices.

Now we can verify Property \eqref{n8}.  Let $\eta_i$ denote $\theta_i + 
\sum_{\sigma\in {\mathfrak B}_i} c_{\xi_i}(\sigma)\bar\sigma $.  All singular simplices
that belong to $\eta_i$ are allowable by \cite[Lemma 2.6]{GBF10}, so it only
remains to check that the singular simplices that belong to $\bd\eta_i$ are 
allowable.  First note that if $\tau$ is non-allowable and belongs to
$\bd\theta_i$ then $\tau$ is an element of ${\mathfrak B}^{k-1}$ (because
$\bd\theta_i\subset U_1$ and $\xi_i$ is
allowable), and we have $\bar{\tau}=\tau$ by Property \eqref{n6}.  This implies
that, modulo allowable singular simplices, we have
\begin{equation}
\label{n12}
\bd\theta_i\equiv \sum_{\tau\in {\mathfrak B}^{k-1}}
\,
c_{\bd\theta_i}(\tau)
\bar\tau.
\end{equation}
Combining equations \eqref{n11} and \eqref{n12} gives
\begin{equation}
\label{n13}
\bd\eta_i\equiv
\sum_{\tau\in {\mathfrak B}^{k-1}}
\,
\left[
c_{\bd\theta_i}(\tau)
+
\sum_{\sigma\in {\mathfrak B}_i} 
c_{\xi_i}(\sigma)
c_{\bd\sigma}(\tau)
\right]
\bar{\tau}.
\end{equation}
If $\tau$ is allowable then all singular simplices belonging to $\bar{\tau}$
are allowable, by \cite[Lemma 2.6]{GBF10}.  If
$\tau$ is not allowable and $\bar{\tau}\neq 0$ then $|\tau|$ must intersect
$U_1$, which implies that the expression in brackets in equation \eqref{n13} is
equal to the coefficient of $\tau$ in $\bd\xi_i$, which is 0 since $\xi_i$ is
allowable.  Thus all singular simplices belonging to $\bd\eta_i$ are allowable,
as required.

This completes the proof of Lemma \ref{n2} for the case $A=\emptyset$.  For the
general case, we are given $\xi_i\in I^{\bar p}C_*(U_i;M)$ for $1\leq i\leq m$
with
\[
\sum \xi_i\in I^{\bar p}C_*(A;M).
\]
Let $U_{m+1}=A$ \red{(recall that $A$ is an open set)} and 
$\xi_{m+1}=-\sum_{i=1}^m \xi_i$.   Applying the case 
already proved to the $(m+1)$-tuple $(\xi_1,\ldots,\xi_{m+1})$, we obtain 
$\eta_i\in I^{\bar
p}C_*(U_i\cap U_1;M)$ for $2\leq i\leq m+1$ with 
\[
\xi_1+\sum_{i=2}^{m+1} \eta_i=0, 
\]
and from this it follows that
\[
\xi_1+\sum_{i=2}^m \eta_i\in I^{\bar p}C_*(A;M)
\]
as required.
\end{proof}

\bibliographystyle{amsplain}
\bibliography{../../bib}

Several diagrams in this paper were typeset using the \TeX\, commutative
diagrams package by Paul Taylor. 
 
\end{document}